\documentclass[11pt]{article}
\usepackage{amsmath,amssymb,amsthm}
\usepackage{amscd}
\newtheorem{theorem}{Theorem}[section]
\newtheorem{lemma}[theorem]{Lemma}
\newtheorem{proposition}[theorem]{Proposition}
\newtheorem{corollary}[theorem]{Corollary}
\theoremstyle{plain}

\theoremstyle{definition}
\newtheorem{definition}[theorem]{Definition}

\numberwithin{equation}{section}

\renewcommand{\labelenumi}{\textup{(\theenumi)}}

\newcommand{\Hom}{\operatorname{Hom}}
\newcommand{\Homeo}{\operatorname{Homeo}}
\newcommand{\id}{\operatorname{id}}
\newcommand{\Ker}{\operatorname{Ker}}

\newcommand{\Ad}{\operatorname{Ad}}

\title{
On extensions of subshifts by finite groups
}
\author{Kengo Matsumoto \\
Department of Mathematics \\
Joetsu University of Education \\
Joetsu, 943-8512, Japan
}

\begin{document}
\maketitle

\def\det{{{\operatorname{det}}}}

\begin{abstract}
$\lambda$-graph systems are labeled Bratteli diagram with shift operations.
They present subshifts.
Their matrix presentations are called symbolic matrix systems. 
We define skew products of $\lambda$-graph systems
and study extensions of subshifts by finite groups.
We prove that two canonical symbolic matrix systems are 
$G$-strong shift equivalent if and only if 
their presented subshifts are  $G$-conjugate.
$G$-equivalent classes of subshifts are classified 
by the cohomology classes of their associated  skewing functions.
  \end{abstract}

 

\def\Re{{\operatorname{Re}}}
\def\det{{{\operatorname{det}}}}
\newcommand{\K}{\mathbb{K}}

\newcommand{\N}{\mathbb{N}}
\newcommand{\C}{\mathbb{C}}
\newcommand{\R}{\mathbb{R}}
\newcommand{\T}{\mathbb{T}}
\newcommand{\Z}{\mathbb{Z}}
\newcommand{\Zp}{{\mathbb{Z}}_+}
\newcommand{\f}{\tilde{f}}

\def\OA{{{\mathcal{O}}_A}}
\def\OB{{{\mathcal{O}}_B}}
\def\OZ{{{\mathcal{O}}_Z}}
\def\OL{{{\mathcal{O}}_{\frak L}}}
\def\OLG{{\mathcal{O}}_{{\frak L}^{G,\ell}}}
\def\FLG{{\mathcal{F}}_{{\frak L}^{G,\ell}}}
\def\DLG{{\mathcal{D}}_{{\frak L}^{G,\ell}}}
\def\SOA{{{\mathcal{O}}_A}\otimes{\mathcal{K}}}
\def\SOB{{{\mathcal{O}}_B}\otimes{\mathcal{K}}}
\def\SOZ{{{\mathcal{O}}_Z}\otimes{\mathcal{K}}}
\def\SOTA{{{\mathcal{O}}_{\tilde{A}}\otimes{\mathcal{K}}}}
\def\DA{{{\mathcal{D}}_A}}
\def\DB{{{\mathcal{D}}_B}}
\def\DZ{{{\mathcal{D}}_Z}}
\def\DTA{{{\mathcal{D}}_{\tilde{A}}}}
\def\SDA{{{\mathcal{D}}_A}\otimes{\mathcal{C}}}
\def\SDB{{{\mathcal{D}}_B}\otimes{\mathcal{C}}}
\def\SDZ{{{\mathcal{D}}_Z}\otimes{\mathcal{C}}}
\def\SDTA{{{\mathcal{D}}_{\tilde{A}}\otimes{\mathcal{C}}}}
\def\Max{{{\operatorname{Max}}}}
\def\Per{{{\operatorname{Per}}}}
\def\PerB{{{\operatorname{PerB}}}}
\def\Homeo{{{\operatorname{Homeo}}}}
\def\HA{{{\frak H}_A}}
\def\HB{{{\frak H}_B}}
\def\HSA{{H_{\sigma_A}(X_A)}}
\def\Out{{{\operatorname{Out}}}}
\def\Aut{{{\operatorname{Aut}}}}
\def\Ad{{{\operatorname{Ad}}}}
\def\Inn{{{\operatorname{Inn}}}}
\def\det{{{\operatorname{det}}}}
\def\exp{{{\operatorname{exp}}}}
\def\cobdy{{{\operatorname{cobdy}}}}
\def\Ker{{{\operatorname{Ker}}}}
\def\ind{{{\operatorname{ind}}}}
\def\id{{{\operatorname{id}}}}
\def\supp{{{\operatorname{supp}}}}
\def\co{{{\operatorname{co}}}}
\def\U{{{\mathcal{U}}}}
\def\LCHD2{{{{\frak L}^{Ch(D_2)}}}}
\def\LCHLA{{{{\frak L}^{Ch(\Sigma_2)}}}}


\def\Zp{{ {\mathbb{Z}}_+ }}
\def\OL{{{\mathcal{O}}_{{\frak L}}}}
\def\M{{{\mathcal{M}}}}
\def\H{{{\mathcal{H}}}}
\def\K{{{\mathcal{K}}}}
\def\P{{{\mathcal{P}}}}
\def\Q{{{\mathcal{Q}}}}
\def\A{{{\mathcal{A}}}}
\def\B{{{\mathcal{B}}}}
\def\R{{{\mathcal{R}}}}
\def\S{{{\mathcal{S}}}}
\def\G{{{\mathcal{G}}}}
\def\sms{{{symbolic  matrix system }}}
\def\smss{{{symbolic  matrix systems }}}
\def\nnms{{{nonnegative matrix system }}}
\def\nnmss{{{nonnegative  matrix systems }}}
\def\Ext{{{\operatorname{Ext}}}}
\def\Hom{{{\operatorname{Hom}}}}
\def\Ker{{{\operatorname{Ker}}}}



\section{Introduction}
Let $\Sigma$ be a finite set, called an alphabet.
Each element of $\Sigma$ is called a symbol or a label.
Let $\Sigma^{\mathbb{Z}}$ 
be the compact Hausdorff space defined by 
the infinite product space
$\prod_{i=-{\infty}}^{\infty}\Sigma_{i}$ 
where $\Sigma_{i} = \Sigma$,
 endowed with the product topology.
 The homeomorphism
$\sigma$ on $\Sigma^{\Z}$ 
given by 
$\sigma((x_i)_{i\in\Z}) = (x_{i+1})_{i\in\Z}$ 
is called the (full) shift.
 Let $\Lambda$ 
 be a shift invariant closed subset of 
$
\Sigma^{\Z}
$ i.e. $\sigma(\Lambda) = \Lambda$.
 The topological dynamical system 
 $(\Lambda, \sigma\vert_{\Lambda})$ is called a subshift.
 We denote $\sigma\vert_{\Lambda}$ by $\sigma$ and
  write the subshift as $\Lambda$ for short.
A subshift is often
called a symbolic dynamical system.
For an introduction to the theory of symbolic dynamical systems, see
\cite{Ki} and \cite{LM}.
Throughout the paper,
$\Zp$ and $\N$ denote 
the set of all nonnegative integers and
the set of all positive integers
respectively.

Let $G$ be a finite group.
Let $A=[A(i,j)]_{i,j=1}^N$ be an  $N\times N$ 
matrix with entries in nonnegative integers
which is called
a nonnegative matrix.
The matrix defines a finite directed graph $\G_A$
with $N$ vertices such that
the number of the edges from $v_i$ to $v_j$ is $A(i,j)$ for $i,j=1,\dots,N$.
Let $E_{A}$ be the edge set of the graph $\G_{A}$.
The shift space
$\Lambda_A$ for the matrix $A$ is defined as the set of biinfinite sequences 
of concatenating edges in $E_A$.
It is a compact subset of $E_A^{\Z}$ with shift homeomorphism
$\sigma$ written $\sigma_A$.
The subshift  
 $(\Lambda_A, \sigma_A)$ 
is called the shift of finite type 
defined by the nonnegative matrix $A$.
It is also called the SFT obtained from the directed graph 
$\G_A$.
Actions of finite groups on SFTs  have been studied by many authors
related to extensions of SFTs
(cf. \cite{AKM}, \cite{BoyleSullivan}, \cite{BoyleSchmieding}, 
\cite{Field}, \cite{FieldNicol}, \cite{Parry}, etc.).
W. Parry showed how to define extensions of SFTs by finite abelian groups
and Theorem 1.1 below.
Suppose that a map
$\ell:E_{A}\rightarrow G$
is given.
Through the map $\ell$,
the matrix $A$ defines  an 
$N\times N$ matrix $A^\ell$ over the semigroup ring $\Zp G$
over $G$.
The function
$\tau_\ell: \Lambda_A \rightarrow G$
defined by
$\tau_\ell((x_n)_{n \in \Z}) = \ell(x_0)\in G$
for $(x_n)_{n \in \Z} \in \Lambda_A$
yields the skew product 
$(G\ltimes \Lambda_A, \tau_\ell\ltimes\sigma_A)$
written $\Lambda_A^{G,\tau_\ell}$, 
which  is an SFT with a continuous $G$-action 
commuting the shift.
The function $\tau_\ell$ is called a skewing function.
Conversely,
any  SFT having a continuous $G$-action commuting with the shift
is constructed by this way.
An SFT with a continuous $G$-action with commuting the shift 
is called a $G$-SFT.
The following theorem was presented 
in Boyle--Sullivan's paper \cite{BoyleSullivan}
as a Parry's result.
\begin{theorem}[{\cite[Proposition 2.7.1]{BoyleSullivan}}] \label{thm:BS}
Let $A$ and $B$ be nonnegative matrices.
Suppose that maps
$\ell_A:E_A\rightarrow G$ and
$\ell_B:E_B\rightarrow G$ are given.
Let 
$\tau_{\ell_A}:\Lambda_A\rightarrow G$
and
$\tau_{\ell_B}:\Lambda_B\rightarrow G$
be the associated skewing functions.
Then the following are equivalent:
\begin{enumerate}
\renewcommand{\labelenumi}{(\arabic{enumi})}
\item $A^{\ell_A}$ and $B^{\ell_B}$ 
are strong shift equivalent over $\Z_+ G$.
\item There is a topological conjugacy 
$\varPhi: \Lambda_A \rightarrow \Lambda_B$
such that $\tau_{\ell_A} $ is cohomologous to $\tau_{\ell_B}\circ \varPhi$ in 
$C(\Lambda_A,G)$. 
\item There is a topological conjugacy between $G$-SFTs 
$\Lambda_A^{G,\tau_{\ell A}}$ and $\Lambda_B^{G,\tau_{\ell B}}$
 commuting with the $G$-actions.
\end{enumerate}
\end{theorem}
In this paper, 
we will generalize the above results on SFTs to general subshifts.
 The author  has  introduced 
notions of 
$\lambda$-graph system and symbolic matrix system
as presentations of subshifts (\cite{DocMath1999}). 
They are generalized notions of 
$\lambda$-graph ($=$ labeled graph) and symbolic matrix.
A $\lambda$-graph system 
$ {\frak L} = (V,E,\lambda,\iota)$
consists of a vertex set 
$V = V_0 \cup V_1\cup V_2\cup\cdots$, an edge set 
$E = E_{0,1}\cup E_{1,2}\cup E_{2,3}\cup\cdots$, 
a labeling map
$\lambda: E \rightarrow \Sigma$
and a surjective map
$\iota(=\iota_{l,l+1}): V_{l+1} \rightarrow V_l$ 
for each
$l=0,1,\dots$
with a certain compatible condition.
A symbolic matrix system 
 $(\M,I)$
over $\Sigma$ consists of
two sequences of rectangular matrices
$(\M_{l,l+1}, I_{l,l+1}), l=0,1,\dots$.
The matrices $\M_{l,l+1}$ 
 have their entries in formal sums of $\Sigma$ and
the matrices $I_{l,l+1}$ have their entries in $\{0,1\}$.
They satisfy the commutation relations:
$
I_{l,l+1} \M_{l+1,l+2} = \M_{l,l+1}I_{l+1,l+2}
$ 
for
$
l =0,1,\dots.
$
It is required that
each row of $I_{l,l+1}$ 
has at least one $1$ and
each column of $I_{l,l+1}$ 
has exactly one $1$.
A $\lambda$-graph system naturally arises  from a symbolic matrix system $(\M,I)$.
The labeled edges from a vertex $v_i^l \in V_l$ 
to a vertex  $v_j^{l+1} \in V_{l+1}$ are given by the 
$(i,j)$-component $\M_{l,l+1}(i,j)$
of  $\M_{l,l+1}$.
The map $\iota( = \iota_{l,l+1} )$ is defined by 
$\iota_{l,l+1}(v_j^{l+1}) = v_i^l$
precisely if  
$I_{l,l+1}(i,j) =1.$
The $\lambda$-graph systems and the symbolic matrix systems 
 are the same objects and give rise to subshifts 
 by gathering label sequences
 appearing in the labeled Bratteli diagrams of the $\lambda$-graph systems.
Let us denote by $\Lambda_{\frak L}$
the subshift presented by the $\lambda$-graph system ${\frak L}$.
Conversely we have a canonical method 
to construct a $\lambda$-graph system and a symbolic matrix system
from an arbitrary subshift \cite{DocMath1999}.
They are called the canonical $\lambda$-graph system
and the canonical symbolic matrix system for subshift $\Lambda$ 
and written as ${\frak L}^\Lambda$ and $(\M^\Lambda, I^\Lambda)$
respectively.

Let $G$ be a finite group.
We call a subshift $(\Lambda,\sigma)$ a $G$-{\it subshift\/}
if there exists an action of $G$ on $\Lambda$ 
which commutes with the shift $\sigma$.
$G$-subshifts $(\Lambda, \sigma)$ and $(\Lambda', \sigma')$
are said to be $G$-{\it conjugate\/}
if there exists a topological conjugacy between them
commuting with their $G$-actions.
For a given function $\tau:\Lambda\rightarrow G$, 
we may consider a subshift $\Lambda^{G,\tau}$
as an extension of  $\Lambda$ by $\tau$,
and know that the subshift $\Lambda^{G,\tau}$ is  a $G$-subshift.
The extension $\Lambda^{G,\tau}$ 
is also called a skew product 
of $\Lambda$ by $\tau$ and written
$(G\ltimes\Lambda, \tau\ltimes\sigma)$.
Conversely, we may show that  
any $G$-subshift $\widetilde{\Lambda}$
is of the form $\Lambda^{G,\tau}$
for some continuous function $\tau:\Lambda\rightarrow G$.
Let 
$\tau: \Lambda \rightarrow G$
and
$\tau': \Lambda' \rightarrow G$
be 
continuous functions.
Then it is easy to see that 
there exist a topological conjugacy
$
\varPhi: \Lambda \rightarrow \Lambda'
$
such that $\tau$ is cohomologous to
$\tau'\circ \varPhi$
if and only if
there exists a topological conjugacy 
between $G$-subshifts
$\Lambda^{G,\tau}
$
and
${\Lambda'}^{G,\tau'}
$
commuting 
with their $G$-actions.

For a $\lambda$-graph system 
$ {\frak L} = (V,E,\lambda,\iota)$
and a finite group $G$,
suppose that a map
$\ell:\Sigma\rightarrow G$
is given. 
 We  may define 
 a $\lambda$-graph system 
${\frak L}^{G,\ell}$
as an extension of ${\frak L}$ by $\ell$,
and show that the
$\lambda$-graph system
 ${\frak L}^{G,\ell}$ has a $G$-action.
A ${\lambda}$-graph system
with $G$-action is called a $G$-${\lambda}$-graph system.
 We will know that a characterization of $G$-$\lambda$-graph system
 (Theorem \ref{thm:extlambda}).
 We show the following:
\begin{theorem}[{Theorem \ref{thm:G-conj}}]
For a $\lambda$-graph system 
$ {\frak L} = (V,E,\lambda,\iota)$
 and a map
$\ell: \Sigma\rightarrow G$,
the subshift
$\Lambda_{{\frak L}^{G,\ell}}$
presented by the $G$-$\lambda$-graph system
${\frak L}^{G,\ell}$
becomes a $G$-subshift
which is  $G$-conjugate to the 
skew product 
$G\ltimes \Lambda_{\frak L}$
defined by the function
$\tau_\ell((x_n)_{n \in \Z}) = \ell(x_0)$
for $(x_n)_{n\in \Z} \in \Lambda_{\frak L}$,
that is
\begin{equation*}
(\Lambda_{{\frak L}^{G,\ell}}, \sigma_{{\frak L}^{G,\ell}})
\cong
(G\ltimes \Lambda_{\frak L}, \tau_\ell\ltimes \sigma_{\frak L}).
\end{equation*}
\end{theorem}
Let
$(\M,I)$
be a \sms over $\Sigma$.
Suppose that a map $\ell:\Sigma\rightarrow G$ is given.
Then $(\M,I)$ 
is naturally regarded as a \sms over $\Zp G$
through the map $\ell$,
denoted by $(\M^{\ell},I)$.
We may give a definition of 
properly $G$-strong shift equivalence 
between two \smss over  $\Zp G$.
We will prove the following theorem as a main result of the paper.  
\begin{theorem}[{Theorem \ref{thm:main}}]
Let $G$ be a finite group.
Let ${\frak L}$ and ${\frak L}'$
be  $\lambda$-graph systems over $\Sigma$
and $\Sigma'$, respectively.
Let $(\M,I)$ and $(\M',I')$ be their associated 
symbolic matrix systems, respectively.
Suppose that maps
$\ell:\Sigma\rightarrow G$
and
 $\ell':\Sigma'\rightarrow G$
are given.
Let $(\M^{\ell},I)$ 
and
$({\M'}^{\ell'},I')$ 
be their  \smss over $\Zp G$ 
through the maps $\ell$ and  $\ell'$
respectively.
Consider the following three conditions:
\begin{enumerate}
\renewcommand{\labelenumi}{(\arabic{enumi})}
\item 
$(\M^{\ell},I)$ and $({\M'}^{\ell'},I')$ 
are properly $G$-strong shift equivalent.     
\item 
There exists a topological conjugacy
$\varPhi:\Lambda_{\frak L}\rightarrow\Lambda_{{\frak L}'}$
such that $\tau_\ell$ is cohomologous to $\tau_{\ell'}\circ\varPhi$
in $C(\Lambda_{\frak L},G)$.
\item 
The $G$-subshifts $\Lambda_{{\frak L}^{G,\ell}}$
and
$\Lambda_{{{\frak L}'}^{G,\ell'}}$
are $G$-conjugate.
\end{enumerate}
Then we have
$$
(1) \Longrightarrow (2) \Longleftrightarrow (3).
$$
If in particular,
${\frak L}$ and ${\frak L}'$ are both the canonical $\lambda$-graph systems,
we have
$
(2) \Longrightarrow (1).
$
\end{theorem}
The equivalence between $(2) \Longleftrightarrow (3)$ is easy.
The other two implications 
$
(1) \Longrightarrow (2)
$
and
$
(2) \Longrightarrow (1)
$
for the canonical $\lambda$-graph systems 
are the main ingredients  of this paper which 
will be proved in Section 6.

We will finally present an example of an action of a finite group to a 
$\lambda$-graph system
which presents a nonsofic subshift called Dyck shift $D_2$,
and study $G$-conjugacy classes of extensions of Markov--Dyck shifts
to give examples of extensions of non sofic subshifts which are not $G$-conjugate
(Corollary \ref{cor:MDG}, Proposition \ref{prop:7.7}). 

\section{Preliminaries}
\subsection{Subshifts}
Let $(\Lambda,\sigma)$ be a subshift over $\Sigma$.
   A finite sequence $\mu = (\mu_1,...,\mu_k) $ of elements $\mu_j \in \Sigma$
 is called a word. 
  We denote by $|\mu|$ the length $k$ of $\mu$. 
  A word $\mu = (\mu_1,...,\mu_k)$ is said to appear in 
$x=(x_i)_{i\in \Z} \in \Sigma^{\Z}$ 
if 
$x_m = \mu_1,..., x_{m+k-1} = \mu_k$  
for some $m \in {\Bbb Z}$.
For a subshift $\Lambda$, 
we denote by 
$B_k(\Lambda)$ the set of all words of length $k$ 
appearing in some $x \in \Lambda$,
where $B_0(\Lambda)$ denotes the empty word.
We set
$B_*(\Lambda) = \cup_{k=0}^\infty B_k(\Lambda)$.
Let us denote by $\Lambda^+$ 
the shift space of the right one-sided subshift for $\Lambda$
which is defined by 
\begin{equation*}
\Lambda^+ = \{ (x_n)_{n \in \N} \mid (x_n)_{n \in \Z} \in \Lambda \}.
\end{equation*} 
For $x = (x_n)_{n \in \N} \in \Lambda^+$ and $l \in \Zp$,
the $l$-predecessor set $\Gamma_l^{-}(x)$
for $x$ is defined by
\begin{equation*}
\Gamma_l^{-}(x) = \{ (\mu_1,\dots,\mu_l) \in B_l(\Lambda) \mid
(\mu_1,\dots,\mu_l, x_1,x_2, \dots ) \in \Lambda^+ \}.
\end{equation*}

\subsection{$\lambda$-graph systems}
A $\lambda$-graph system 
is a graphical object presenting a subshift (\cite{DocMath1999}). 
It is a generalization of a finite labeled graph and has a close relationship
to a construction of a certain class of $C^*$-algebras (\cite{DocMath2002}).  
Let ${\frak L} =(V,E,\lambda,\iota)$ be 
a $\lambda$-graph system 
 over $\Sigma$ with vertex set
$
V = \cup_{l \in \Zp} V_{l}
$
and  edge set
$
E = \cup_{l \in \Zp} E_{l,l+1}
$
with a labeling map
$\lambda: E \rightarrow \Sigma$, 
and that is supplied with  surjective maps
$
\iota( = \iota_{l,l+1}):V_{l+1} \rightarrow V_l
$
for
$
l \in  \Zp.
$
Here the vertex sets $V_{l},l \in \Zp$
are finite disjoint sets.   
Also  
$E_{l,l+1},l \in \Zp$
are finite disjoint sets.
An edge $e$ in $E_{l,l+1}$ has its source vertex $s(e)$ in $V_{l}$ 
and its terminal  vertex $t(e)$ 
in
$V_{l+1}$
respectively.
Every vertex in $V$ has a successor and  every 
vertex in $V_l$ for $l\in {\N}$ 
has a predecessor. 
It is then required that there exists an edge in $E_{l,l+1}$
with label $\alpha$ and its terminal is  $v \in V_{l+1}$
 if and only if 
 there exists an edge in $E_{l-1,l}$
with label $\alpha$ and its terminal is $\iota(v) \in V_{l}.$
For 
$u \in V_{l-1}$ and
$v \in V_{l+1},$
put
\begin{align}
E^{\iota}(u, v)
& = \{e \in E_{l,l+1} \ | \ t(e) = v, \iota(s(e)) = u \}, \label{eq:2.1}\\
E_{\iota}(u, v)
& = \{e \in E_{l-1,l} \ | \ s(e) = u, t(e) = \iota(v) \}. \label{eq:2.2}
\end{align}
Then we require a bijective correspondence preserving their labels between 
$
E^{\iota}(u, v)
$
and
$
E_{\iota}(u, v)
$
for each pair of vertices
$u, v$.
We call this property {\it  the local property\/} of $\lambda$-graph system. 
We call an edge in $E$ a labeled edge,
and a finite sequence of connecting labeled edges a labeled path or a $\lambda$-path.
If a labeled path $\gamma$ labeled $\nu$
starts at a vertex $v$ in $V_l$
and ends at a vertex $u$ in $V_{l+n}$,
we say that $\nu$ leaves $v$
and write 
$s(\gamma)=v, t(\gamma) = u, \lambda(\gamma) = \nu.$  
We henceforth assume that ${\frak L}$ is left-resolving, 
which means that 
$t(e)\ne t(f)$ whenever $\lambda(e) = \lambda(f)$, $e\ne f$
 for $e,f \in E$.
For a vertex
$v \in V_l$ denote by
$\Gamma_l^-(v)$ the predecessor set of $v$
which is defined by the set of  words  of length $l$
appearing as labeled paths from a vertex in $V_0$ to the vertex $v$.  
 ${\frak L}$ is said to be predecessor-separated if 
$\Gamma_l^-(v) \ne \Gamma_l^-(u)$
whenever $u, v\in V_l$  are distinct.
A subshift $\Lambda$ is said to be presented by 
a $\lambda$-graph system 
${\frak L}$
if the set of admissible words of $\Lambda$
coincides with the set of labeled paths appearing somewhere in ${\frak L}$.
$\lambda$-graph systems
${\frak L} =(V,E,\lambda,\iota)$ over $\Sigma$
and
${\frak L}' =(V',E',\lambda',\iota')$ over $\Sigma'$
are said to be isomorphic if there exist
bijections
$\varPhi_V:V \rightarrow V',$
$\varPhi_E:E \rightarrow E'$
and
$\phi:\Sigma \rightarrow \Sigma'$
satisfying
$\varPhi_V(V_l) = V'_l,$
$\varPhi_E(E_{l,l+1}) = E'_{l,l+1}$
and
$\lambda' \circ \varPhi_E = \phi\circ \lambda$
such that
they give rise to a labeled graph isomorphism
compatible to $\iota$ and $ \iota'$.
We note that any essential finite directed labeled graph 
${\cal G} = ({\cal V}, {\cal E}, \lambda)$ over $\Sigma$
with vertex set ${\cal V}$, 
edge set ${\cal E}$ and labeling map $\lambda:{\cal E}\longrightarrow \Sigma$
 gives rise to a $\lambda$-graph system
${\frak L}_{\cal G} =(V,E,\lambda,\iota)$
by setting
$V_l ={\cal V}, E_{l,l+1} ={\cal E}, \iota = \id$
for all $l \in \Zp$
(cf. \cite{DocMath2002}).

Two points $x, y \in \Lambda^+$ are said to be 
$l$-{\it past equivalent},
written as
$x \sim_l y$,
if $\Gamma_l^-(x) = \Gamma_l^-(y)$.
For a fixed 
$l \in \Zp$, 
let 
$F_i^l, i= 1, 2,\dots, m(l)$ 
be the set  of all $l$-past equivalence classes of $\Lambda^+$
so that $\Lambda^+$ is a disjoint union of 
 $F_i^l, i= 1, 2,\dots, m(l)$. 
 Then the canonical $\lambda$-graph system
$
{\frak L}^\Lambda =(V^\Lambda,E^\Lambda,\lambda^\Lambda,\iota^\Lambda )
$
for $\Lambda$ is defined as follows
(\cite{DocMath1999}). 
The vertex set $V_l^\Lambda$ 
at level $l$
consists of the sets 
 $F_i^l,i=1,\dots,m(l)$.
 We write an edge with label $\alpha$ 
    from the vertex 
 $F_i^l \in V_l^\Lambda$ 
 to the vertex
 $F_{j}^{l+1} \in V_{l+1}^\Lambda$ 
 if
$ \alpha x \in F_i^l$
 for some 
 $x \in F_j^{l+1}$.
We denote by $E_{l,l+1}^\Lambda$ 
the set of all edges from $V_l^\Lambda$ 
to $V_{l+1}^\Lambda$.
There exists a natural map $\iota^{\Lambda}_{l,l+1}$ 
from $V_{l+1}^\Lambda$ 
to
$V_l^\Lambda$ 
by mapping $F_j^{l+1}$ to $F_i^l $
 when 
  $F_i^l $ contains  $F_j^{l+1}$.
Set
$V^\Lambda = \cup_{l \in \Zp} V_l^\Lambda $
 and
$E^\Lambda = \cup_{l \in \Zp} E_{l,l+1}^\Lambda$.
The labeling of edges is denoted by
$\lambda^\Lambda:E^\Lambda \rightarrow \Sigma$.
 The canonical $\lambda$-graph system
$
{\frak L}^\Lambda 
$
is left-resolving and predecessor-separated, 
and presents $\Lambda$.

For a $\lambda$-graph system ${\frak L}$,
let $\Lambda_{\frak L} $
be the presented subshift by
${\frak L}$.
Then its canonical $\lambda$-graph system 
$
{\frak L}^{\Lambda_{\frak L}} 
$
does not necessarily coincide with the original
$\lambda$-graph system
${\frak L}$.
If in particular,
 $\Lambda$ is a sofic shift,
its canonical $\lambda$-graph system  is eventually realized  as the 
left Krieger cover graph for $\Lambda$.
\subsection{Symbolic matrix systems}
For an alphabet $\Sigma$, 
let us denote by ${\frak S}_{\Sigma}$
the set of formal sums of elements of $\Sigma$.
It contains $0$ as an empty word $\emptyset$.
A symbolic matrix system is a matrix presentation
of a $\lambda$-graph system.
It consists of
 a pair  $(\M_{l,l+1},I_{l,l+1}), l \in \Zp$
of sequences of rectangular matrices 
such that the following conditions for each $l \in \Zp$
are satisfied:
\begin{enumerate}
\item
$\M_{l,l+1} $ is an $m(l) \times m(l+1)$ rectangular matrix with entries in 
${\frak S}_{\Sigma}$.
\item
$I_{l,l+1} $ is an $m(l) \times m(l+1)$ rectangular matrix with entries in 
$\{0,1 \}$ 
satisfying the relation:
\begin{equation}
I_{l,l+1}\M_{l+1,l+2} = \M_{l,l+1} I_{l+1,l+2},
\qquad l \in \Zp.
\end{equation}
\end{enumerate}
We further assume that
both the matrices $\M_{l,l+1}$ and $I_{l,l+1}$ have no zero columns and no zero rows.
For $j$, 
there uniquely exists $i$ such that
$I_{l,l+1}(i,j) \ne 0$.
By the above conditions one sees 
$m(l) \le m(l+1).$
The pair $(\M,I)$ is called 
a {\it symbolic matrix system} over $\Sigma$.

Symbolic matrix systems 
$(\M,I)$ over  $\Sigma$
and
$(\M',I')$ over  $\Sigma'$
are said to be isomorphic 
if $m(l) = m^\prime(l)$ for $l \in \Zp$ 
and
there exists a specification $\phi$ 
from $\Sigma$ to $\Sigma'$ 
and an $m(l) \times m(l)$ permutation matrix
$P_l$ for each $l \in \Zp$ 
such that
$$
P_l \M_{l,l+1} \overset{\phi}{\simeq} \M'_{l,l+1}P_{l+1},
\qquad
P_l I_{l,l+1} = I'_{l,l+1}P_{l+1}
\qquad
\text{ for }
\quad l \in \Zp
$$
where specification $\phi:\Sigma\rightarrow \Sigma'$ 
means a bijective relabeling map,
and 
$\overset{\phi}{\simeq}$
 means the entrywise equalities through $\phi$.
Symbolic matrix systems 
$(\M,I)$ 
and
$(\M',I')$ 
are said to be shift isomorphic 
if there exist $k,k' \in \Zp$ such that 
\begin{equation*}
\M_{l+k,l+k+1}
=
\M'_{l+k',l+k'+1}, \qquad
I_{l+k,l+k+1}
=
I'_{l+k',l+k'+1} \quad
\text{ for all } l \in \Zp
\quad
(\cite{IJM2004}).
\end{equation*}
The notion of \sms is a generalized notion of symbolic matrix.
For an $n \times n$ symbolic matrix $\A$, we set
$\M_{l,l+1} = \A,  I_{l,l+1} = E_n
$
for
$l \in \Zp$
where 
$E_n$ denotes the identity matrix of size $n$.
Then 
 $(\M_{l,l+1},I_{l,l+1}), l \in \Zp$
 is a symbolic matrix system.
 There exists a natural bijective correspondence 
between 
the set of isomorphism classes of \smss 
and 
the set of isomorphism classes of $\lambda$-graph systems.
We say a \sms to be canonical for a subshift $\Lambda$ if
its corresponding $\lambda$-graph system is canonical.
It is denoted by
$(\M^\Lambda, I^\Lambda)$.

\subsection{Strong shift equivalence of symbolic matrix systems}
The notion of strong shift equivalence in nonnegative matrices 
is a fundamental notion in the classification theory of shifts of finite type.
It has  been introduced by R. F. Williams \cite{Wi} to classify 
shifts of finite type by topological conjugacy
 (for sofic shifts see \cite{Nasu}).
As a generalization of Williams's strong shift equivalence,
two kinds of  strong shift equivalences 
between symbolic matrix systems
have been introduced in \cite{DocMath1999}.
One is called the properly strong shift equivalence that 
exactly reflects a bipartite decomposition of the associated 
$\lambda$-graph systems.
The other one is called the strong shift equivalence
that is  weaker than the former strong shift equivalence.
They coincide at least among \smss whose $\lambda$-graph systems are
left-resolving and predecessor-separated,
and hence between canonical \smss for subshifts.
  The latter is easier defined and treated than the former
(see \cite{DocMath1999} for the detail).

Let us recall the two strong shift equivalences in symbolic matrix systems.
For alphabets $C,D$,
put
$C \cdot D = \{ cd \mid c \in C,d \in D \}.$
For 
$x = \sum_{j}c_j \in {\frak S}_C$
and
$ y = \sum_{k} d_k \in {\frak S}_D$,
 define
 $xy = \sum_{j,k} c_jd_k \in {\frak S}_{C\cdot D}$.
Let $(\M,I)$ and $(\M^{\prime},I')$ be \smss over 
$\Sigma$ and $\Sigma^{\prime}$ respectively,
where
$\M_{l,l+1},I_{l,l+1}$ are $m(l)\times m(l+1)$ matrices and
$\M'_{l,l+1},I'_{l,l+1}$ are $m'(l)\times m'(l+1)$ matrices.
Symbolic matrix systems 
 $(\M,I)$ and $(\M^{\prime},I')$
 are said to be {\it properly strong shift equivalent in 1-step},
written as 
$
(\M,I)\underset{1-pr}{\approx} (\M',I'),
$
 if
 there exist alphabets 
$C,D$ and
 specifications 
$
 \varphi: \Sigma \rightarrow C\cdot D,
\phi: \Sigma' \rightarrow D \cdot C
$
 and increasing sequences $n(l),n'(l)$ on $l \in \Zp$
 such that
 for each $l\in \Zp$, 
there exist
an $n(l)\times n'(l+1)$ matrix ${\P}_l$ over $C,$ 
an $n'(l)\times n(l+1)$ matrix ${\Q}_l$ over $D,$
an $n(l)\times n(l+1)$ matrix $X_l $ over $ \{0,1\}$
and 
an $n'(l)\times n'(l+1)$ matrix $X'_l$ over $ \{0,1\}$
 satisfying the following equations: 
\begin{align}
{\M}_{l,l+1} 
\overset{\varphi}{\simeq} {\P}_{2l}{\Q}_{2l+1},& \qquad
{\M'}_{l,l+1} 
\overset{\phi}{\simeq} {\Q}_{2l}{\P}_{2l+1},\label{eq:MPQ}\\
I_{l,l+1} = X_{2l}X_{2l+1},& \qquad 
I'_{l,l+1} = X'_{2l}X'_{2l+1}  \label{eq:IXX}\\
\intertext{and}
X_l{\P}_{l+1}= {\P}_{l}{X'}_{l+1},& \qquad
X'_l{\Q}_{l+1}= {\Q}_{l}{X}_{l+1}, \label{eq:XP}
\end{align}
where 
$\overset{\varphi}{\simeq}$
and
$\overset{\phi}{\simeq}$
mean the equalities through the specifications
$\varphi, \phi$, respectively. 
It follows by \eqref{eq:MPQ} that
$n(2l) = m(l)$ and $ n'(2l) = m(l)$ for $l \in \Zp$.

Two \smss 
 $(\M,I)$ and $(\M^{\prime},I')$ 
 are said to be {\it properly strong shift equivalent in N-step},
written as
$
(\M,I) \underset{N-pr}{\approx} (\M',I'),
$
 if
there exists an $N$-string of properly strong shift equivalences in 1-step
of \smss connecting between 
$(\M,I)$
and
$
(\M',I').
$
We simply call it a {\it properly strong shift equivalence}.
The properly strong shift equivalence is rephrased in terms of 
 bipartite symbolic matrix systems
and bipartite $\lambda$-graph systems.
Suppose that
$
(\M,I) \underset{1-pr}{\approx} (\M',I'),
$
Let $\Lambda$ and $\Lambda'$ be their presenting subshifts.
For $(x_n)_{n \in \Z} \in \Lambda$,
we write $ \varphi(x_n) =c_n d_n\in C\cdot D, n \in \Z.$
Put $y_n = \phi^{-1}(d_n c_{n+1})$.
As in  \cite{DocMath1999}, one knows that 
$(y_n)_{n \in \Z}$ defines an element of $\Lambda'$
such that the correspondence 
$\varPhi:
(x_n)_{n \in \Z} \in \Lambda 
\rightarrow
(y_n)_{n \in \Z} \in \Lambda'
$
yields a topological conjugacy
from $\Lambda$ to $\Lambda'$,
which is called the forward bipartite conjugacy.
If one takes $y'_n = \phi^{-1}(d_{n-1}c_n), $
the topological conjugacy
$\varPhi':
(x_n)_{n \in \Z} \in \Lambda 
\rightarrow
(y'_n)_{n \in \Z} \in \Lambda'
$
is called the backward bipartite conjugacy (cf. \cite{Nasu}).
Hence
if two  \smss are properly strong shift equivalent,
their presented subshifts  are topologically conjugate. 
Conversely,
if two subshifts $\Lambda,\Lambda'$
 are topologically conjugate,
their canonical  \smss
$(\M^{\Lambda},I^{\Lambda}),(\M^{\Lambda'},I^{\Lambda'})$ 
are
properly strong shift equivalent (\cite{DocMath1999}).

The above definition of properly strong shift equivalence for \smss
needs 
rather complicated formulations than that of strong shift equivalence
for nonnegative matrices.
The notion of  strong shift equivalence 
between two \smss 
is simpler 
and weaker than properly strong shift equivalence.
Symbolic matrix systems $(\M,I)$ and $(\M',I')$ 
 are said to be {\it  strong shift equivalent in 1-step},
written as
$
(\M,I) \underset{1-st}{\approx}(\M',I'),
$
if
 there exist alphabets 
$C,D$ and
 specifications 
$
 \varphi: \Sigma \rightarrow C\cdot D,
$
$
 \phi: \Sigma' \rightarrow D\cdot C
$
 such that
 for each $l\in \N$, 
 there exist
an $ m(l-1)\times m'(l)$ matrix $ {\H}_l $ over $ C $
 and
an $ m'(l-1)\times m(l)$ matrix $ {\K}_l $ over $ D $ 
 satisfying the following equations: 
\begin{align}
I_{l-1,l}{\M}_{l,l+1} 
\overset{\varphi}{\simeq} {\H}_l{\K}_{l+1}, & \qquad
I^{\prime}_{l-1,l}{\M}^{\prime}_{l,l+1} 
\overset{\phi}{\simeq} {\K}_l{\H}_{l+1} \label{eq:IMHK}\\
\intertext{and}
{\H}_lI^{\prime}_{l,l+1}= I_{l-1,l}{\H}_{l+1},& \qquad
{\K}_lI_{l,l+1}= I^{\prime}_{l-1,l}{\K}_{l+1}.\label{eq:HIIH}
\end{align}
Symbolic matrix systems 
 $(\M,I)$ and $(\M^{\prime},I')$ 
 are said to be {\it  strong shift equivalent in N-step},
written as
$
(\M,I) \underset{N-st}{\approx} (\M',I'),
$
 if
there exists an $N$-string of strong shift equivalences of 1-step of \smss
connecting between
$(\M,I)$
and
$ 
(\M', I').
$
We simply call it a {\it  strong shift equivalence}.
Let
$\P_l,\Q_l, X_l$ and $ X'_l$ be the matrices 
in the definition of properly  strong shift equivalence in 1-step
between 
$(\M,I) $
and
$(\M',I')$
satisfying \eqref{eq:MPQ}, \eqref{eq:IXX} and \eqref{eq:XP}.
By setting 
$$
\H_l = X_{2l-1}\P_{2l-1},\qquad
\K_l = X'_{2l-1}\Q_{2l-1},
$$
they give rise to a strong shift equivalence in 1-step between
$(\M,I)$ and $(\M',I')$.
Hence properly strong shift equivalence in 1-step implies
strong shift equivalence in 1-step.
Let
$(\M,I)$ and 
$(\M',I')$ be the symbolic matrix systems for 
$\lambda$-graph systems ${\frak L}$ and ${\frak L}^{\prime}$
respectively.
Suppose that both ${\frak L}$ and ${\frak L}^{\prime}$ are left-resolving,
and predecessor-separated.
We  know  that 
$(\M,I)$ and $(\M',I')$ are strong shift equivalent in $l$-step     
if and only if  
$(\M,I)$ and $(\M',I')$ are properly strong shift equivalent in $l$-step
(\cite{ETDS2003}).     
Hence the two notions,  
strong shift equivalence
 and
 properly strong shift equivalence, 
coincide with
each other in the canonical symbolic matrix systems.   
We finally note that shift isomorphisms between
\smss imply strong shift equivalent (\cite[Proposition 2.2]{IJM2004}).

\section{Extensions of subshifts}
Let $(\Lambda, \sigma)$
be a subshift over $\Sigma$
and $G$ a finite group.
Let $\tau:\Lambda\rightarrow G$
be a continuous function from $\Lambda$ to $G$.
Define a homeomorphism
$\sigma_{G,\tau}:G\times \Lambda \rightarrow G\times\Lambda$
by setting
$\sigma_{G,\tau}(g,x) = (g\tau(x),\sigma(x))$.
We set
$\Lambda^{G,\tau} = G \times \Lambda$.
\begin{lemma}
$(\Lambda^{G,\tau},\sigma_{G,\tau})$ 
is topologically conjugate to a subshift.
\end{lemma}
\begin{proof}
Define a metric $d$ on $\Lambda$ by
\begin{equation*}
d(x,y) =
\begin{cases}
2^{-k} & \text{ if } x \ne y \text{ and } 
     k = \Max\{n \in \Zp \mid x_{|m|} = y_{|m|} \text{for all } m\in \Z 
         \text{ with } |m| \le n \}, \\
0 & \text{ if } x =y,
\end{cases} 
\end{equation*}
which gives rise to a topology equivalent to the original product 
topology on $\Lambda$.
Consider the metric
on $G \times \Lambda$ induced by $d$ on $\Lambda$ in a natural way.
Then it is easy to see that the homeomorphism
$\sigma_{G,\tau}:G\times \Lambda \rightarrow G\times\Lambda$
is expansive relative to the metric on $G \times \Lambda$.
Hence the topological dynamical system
$(\Lambda^{G,\tau},\sigma_{G,\tau})$ 
is topologically conjugate to a subshift.
\end{proof}
The space $\Lambda^{G,\tau} = G \times \Lambda$ 
has a natural $G$-action written $\rho$ from the left
$$
\rho_g(a,x) =(ga,x), \qquad a,g \in G, x \in \Lambda
$$
which commutes with $\sigma$.
Hence we have a $G$-subshift
$(\Lambda^{G,\tau},\sigma_{G,\tau})$.
We call $(\Lambda^{G,\tau},\sigma_{G,\tau})$
the {\it extension\/} of $\Lambda$ by skewing function
$\tau:\Lambda\rightarrow G$.
It is also called the
$G$-extension of $\Lambda$ by $\tau$ for brevity.
From the view point of topological dynamical systems,
the dynamical system
$(\Lambda^{G,\tau},\sigma_{G,\tau})$
is called the skew product written 
 $(G\ltimes\Lambda,\tau\ltimes\sigma).$

Conversely,
let $\widetilde{\Lambda}$ be a subshift with
a $G$-action $\tilde{\rho}$
commuting with the shift $\tilde{\sigma}$. 
Assume that $G$ acts on $\widetilde{\Lambda}$ freely.
Let $q:\widetilde{\Lambda}\rightarrow\Lambda$ 
be the map onto the quotient space
$\Lambda = \widetilde{\Lambda}/G$ of $G$-orbits.
Let $\sigma$ be the homeomorphism on $\Lambda$
induced by $\tilde{\sigma}$.
Since the action of $G$ on $\widetilde{\Lambda}$ is free,
there exists a continuous cross section
$c:\Lambda\rightarrow \widetilde{\Lambda}$ 
such that 
the set
$$
C = \{ c(x) \in \widetilde{\Lambda} \mid x \in \Lambda\}
$$
is a clopen subset of $\widetilde{\Lambda}$ and 
$\widetilde{\Lambda}$ is homeomorphic
to the union $\cup_{g \in G} \tilde{\rho}_g(C)$
which is mutually disjoint.
Hence one may identify
$\widetilde{\Lambda}$ with
$G\times\Lambda$ through
$(g,x) \in G\times\Lambda 
            \rightarrow 
\tilde{\rho}_g(c(x)) \in \widetilde{\Lambda}$.
The cross section
$c:\Lambda \rightarrow \widetilde{\Lambda} = G \times \Lambda$ is 
identified with
$c(x) = (1,x)$.
Define $\tau : \Lambda \rightarrow G$
by setting
$\tau(x) = g$ if $\tilde{\sigma}(c(x))\in \tilde{\rho}_g(C)$.
Since $\tilde{\rho}_g, g \in G$ commute with $\tilde{\sigma}$,
we have
$q(\tilde{\sigma}(x)) = \sigma(q(x))$
for $x \in \widetilde{\Lambda}$
so that
$\tilde{\sigma}(g,x) = (g\tau(x),\sigma(x))$.
Therefore $(\widetilde{\Lambda}, \tilde{\sigma})$
is a skew product
$(G\ltimes\Lambda,\tau\ltimes\sigma)$.
\begin{proposition}[{cf. \cite{AKM}, \cite{BoyleSullivan}, \cite{FieldNicol}}]
Let $\Lambda$ be a subshift
and $G$ a finite group.
Any extension $(\Lambda^{G,\tau},\sigma_{G,\tau})$
of $\Lambda$ by a skewing function $\tau:\Lambda\rightarrow G$ 
defines a $G$-subshift.
Conversely,
any $G$-subshift $(\widetilde{\Lambda}, \tilde{\sigma})$
comes from a skew product $(G\ltimes\Lambda,\tau\ltimes\sigma)$
of a subshift $(\Lambda,\sigma)$ with a skewing function $\tau$.
\end{proposition}
Concerning on topological conjugacy of extensions,
the following proposition is elementary and folklore.
We give its proof for the sake of completeness.
We always assume that $G$ is a finite group.
\begin{proposition}\label{prop:G-actioncoho}
Let
$(\Lambda,\sigma)$
and
$(\Lambda', \sigma')$
be subshifts.
Let 
$\tau: \Lambda \rightarrow G
$
and
$
\tau': \Lambda' \rightarrow G$ 
be continuous functions.
Let us denote by 
$\Lambda^{G,\tau}$
and
$
{\Lambda'}^{G,\tau'}
$
their  $G$-extensions. 
Then the following are equivalent:
\begin{enumerate}
\renewcommand{\labelenumi}{(\arabic{enumi})}
\item There exists a topological conjugacy 
$
\Psi: \Lambda^{G,\tau} \rightarrow {\Lambda'}^{G,\tau'}
$
commuting 
with their $G$-actions.
\item
There exist a topological conjugacy
$
\varPhi: \Lambda \rightarrow \Lambda'
$
and a continuous function
$\gamma:\Lambda\rightarrow G$
such that 
\begin{equation*}
\tau(x)
=\gamma(x)\tau'(\varPhi(x))\gamma(\sigma(x))^{-1},
\qquad
x \in \Lambda.
\end{equation*}
\end{enumerate}
\end{proposition}
\begin{proof}
(1) $\Longrightarrow$ (2):
For 
$(1,x) \in G\times \Lambda= \Lambda^{G,\tau}$, 
we set 
$$
\Psi(1,x)
=(\gamma(x),\varPhi(x)) \in G\times \Lambda'= {\Lambda'}^{G,\tau'} 
$$
for some continuous functions
$\gamma:\Lambda \rightarrow G$
and
$\varPhi:\Lambda \rightarrow \Lambda'$.
Since $\rho'_g \circ \Psi = \Psi \circ \rho_g$
for $g \in G$,
we have 
$
\Psi(g,x)
= (g\gamma(x),\varPhi(x)).
$
Since 
$\sigma'_{G,\tau'} \circ \Psi = \Psi\circ \sigma_{G,\tau}$
and
\begin{gather*}
\sigma'_{G,\tau'}\circ \Psi(g,x)
= \sigma'_{G,\tau'} (g\gamma(x),\varPhi(x)) 
= (g\gamma(x)\tau'(\varPhi(x)), \sigma'(\varPhi(x)) ), \\
\Psi\circ\sigma_{G,\tau}(g,x)
=\Psi(g\tau(x),\sigma(x))
= ( g\tau(x)\gamma(\sigma(x)), \varPhi(\sigma(x)) ),
\end{gather*}
we have 
\begin{equation*}
(g\gamma(x)\tau'(\varPhi(x)), \sigma'(\varPhi(x)))
=   (g\tau(x)\gamma(\sigma(x)),
\varPhi(\sigma(x))).
\end{equation*}
Hence we have
\begin{equation*}
\gamma(x)\tau'(\varPhi(x))
=\tau(x)\gamma(\sigma(x)),
\qquad
\sigma'(\varPhi(x))
=
\varPhi(\sigma(x))
\end{equation*}
which show that 
$\tau$ is cohomologous to
$\tau'\circ \varPhi$
and
$
\varPhi: \Lambda \rightarrow \Lambda'
$
is a topological conjugacy.

(2) $\Longrightarrow$ (1): 
Define
$
\Psi: \Lambda^{G,\tau} \rightarrow {\Lambda'}^{G,\tau'}
$
by setting
$$
\Psi(g, x) 
= (g \gamma(x),\varPhi(x))
\quad 
\text{ for }
(g, x) \in G \times \Lambda =\Lambda^{G,\tau}.
$$
It then follows that 
for $x =(x_n)_{n \in \Z} \in \Lambda$
\begin{equation*}
\Psi(\sigma_{G,\tau}(g,x))
  = (g \tau(x)\gamma(\sigma(x)), \varPhi(\sigma(x))),
  \qquad
\sigma_{G,\tau'}(\Psi(g,x)) 
= (g \gamma(x)\tau'(\varPhi(x)), \sigma'(\varPhi(x)))
\end{equation*}
so that
$
\Psi(\sigma_{G,\tau}(g,x)) = \sigma_{G,\tau'}(\Psi(g,x)).
$
It is routine to check that 
$\Psi$ is a homeomorphism 
so that 
$\Psi:\Lambda^{G,\tau}\rightarrow {\Lambda'}^{G,\tau'}$ 
yields a topological conjugacy.
The equality 
$\Psi \circ \rho_g = \rho'_g \circ \Psi$
for $g \in G$
is clear.
Hence 
$
\Psi: \Lambda^{G,\tau} \rightarrow {\Lambda'}^{G,\tau'}
$
gives rise to a $G$-conjugacy.
\end{proof}

\section{Extensions of $\lambda$-graph systems}
In what follows, 
we fix a finite group  $G$
and 
a $\lambda$-graph system 
 ${\frak L}=(V,E,\lambda,\iota)$
over an alphabet $\Sigma$.
Suppose that a map
$\ell:\Sigma \rightarrow G$
is given.
We set
$\Sigma^G = G \times \Sigma$.
We will construct 
a $\lambda$-graph system 
 ${\frak L}^{G,\ell}=(V^G, E^G, \lambda^G, \iota^G)$
over $\Sigma^G$
from ${\frak L}$ by the map $\ell$,
which we call an 
{\it extension of\/}${\frak L}$ by $G$,
or a $G$-extension of ${\frak L}$
by $\ell$.
We put 
$\ell_G = \ell\circ\lambda:E\rightarrow G$.
For $l \in \Zp$,
we set
$V_l^G = G \times V_l$
and
$E_{l,l+1}^G = G \times E_{l,l+1}$.
Define
$\iota^G:V_{l+1}^G\rightarrow V_l^G$
 by setting
$\iota^G(g,v) = (g,\iota(v))$ for 
$(g,v) \in V_{l+1}^G$.
We put
$e^g = (g,e) \in E_{l,l+1}^G$.
The source $s(e^g)$
and the terminal $t(e^g)$
of $e^g$
are defined respectively by
\begin{equation*}
s(e^g)=(g,s(e))\in V_l^G, \qquad
t(e^g)=(g \ell_G(e), t(e))\in V_{l+1}^G.
\end{equation*}
Define
$\lambda^G:E_{l,l+1}^G \rightarrow \Sigma^G$
by
$\lambda^G(e^g) = (g,\lambda(e))$.
We set
$V^G=\cup_{l=0}^\infty V_l^G$
and
$E^G=\cup_{l=0}^\infty E_{l,l+1}^G$.
We then have
\begin{proposition}
${\frak L}^{G,\ell}=(V^G, E^G, \lambda^G, \iota^G)$
is a $\lambda$-graph system over $\Sigma^G$.
\end{proposition}
\begin{proof}
It suffices to show the local property of $\lambda$-graph system
for
${\frak L}^G$.
Take an arbitrary   
$(a,u) \in V_{l-1}^G$, $(b,v) \in V_{l+1}^G$
and fix them.
For 
$e^a\in E_{\iota}((a,u), (b,v))$,
we have 
$s(e^a) = (a,u), t(e^a) = (b,\iota(v))$. 
Hence $b = a\ell_G(e)$.
By the local property of ${\frak L}$,
one may find $f \in E^\iota(u,v)$ such that
$\iota(s(f)) = u, t(f) = v$ and $\lambda(f) = \lambda(e)$
so that $\ell_G(f) = \ell_G(e)= a^{-1}b$.
The edge $f^a =(a,f)$ belongs to
$ E^{\iota}((a,u), (b,v))$ and $\lambda^G(f^a) = \lambda^G(e^a)$.
Conversely,
for any edge $f^a \in E^{\iota}((a,u), (b,v))$,
one may find an edge 
$e^a\in E_{\iota}((a,u), (b,v))$
such that 
 $\lambda^G(f^a) = \lambda^G(e^a)$.
Hence 
there exists a bijective correspondence between
$E_\iota((a,u),(b,v))$ and 
$E^\iota((a,u),(b,v))$ preserving their labels.
\end{proof}
We say that a finite group $G$
acts on a $\lambda$-graph system
${\frak L} = (V,E,\lambda,\iota)$ 
 over $\Sigma$ if
there exists a triplet of bijective maps
$\rho_g =(\rho_g^V, \rho_g^E, \rho_g^\Sigma)$
for each $g \in G$
with 
$
\rho^V_g:V_l \rightarrow V_l,
\rho^E_g:E_{l,l+1} \rightarrow E_{l,l+1}
$
and
$\rho^\Sigma_g:\Sigma \rightarrow \Sigma
$
such that
$\rho^V_g\circ \iota
=\iota\circ\rho^V_g$  
and
$$
\rho^V_g(s(e)) = s(\rho^E_g(e)),\qquad
\rho^V_g(t(e)) = t(\rho^E_g(e)),\qquad
\rho^\Sigma_g(\lambda(e)) = \lambda(\rho^E_g(e)),\qquad
e \in E
$$
and
$
 \rho_1^* = \id_*, 
\rho_{g_1}^*\circ\rho_{g_2}^* = \rho_{g_1 g_2}^*,
 g_1, g_2 \in G
$
for 
$* = V, E, \Sigma$,
respectively.
A $\lambda$-graph system with an action of $G$ is called a $G$-$\lambda$-{\it graph system}.
We may define $G$-action $\rho$ on ${\frak L}^{G,\ell}$.
For $g \in G$, define
$\rho^V_g:V_l^G \rightarrow V_l^G,$
$\rho^E_g:E_{l,l+1}^G \rightarrow E_{l,l+1}^G$
and
$\rho^{\Sigma^G}_g:\Sigma^G \rightarrow \Sigma^G$
by 
$\rho^V_g(a,v) = (ga,v) \in V_l^G$,
$\rho^E_g(a,e) = (ga,e) \in E_{l,l+1}^G$
and
$\rho^{\Sigma^G}_g(a,\alpha) = (ga,\alpha) \in \Sigma^G$
for $g, a \in G, v\in V, e \in E, \alpha\in \Sigma$,
respectively.
It is easy to see 
$\rho_g \circ \rho_h = \rho_{gh}$.
Hence the $G$-extension
${\frak L}^{G,\ell}$ is a $G$-$\lambda$-graph system.

The correspondences
\begin{equation*}
q_{V^G}: V_l^G\rightarrow V_l, 
\qquad
q_{E^G}: E_{l,l+1}^G\rightarrow E_{l,l+1},
\qquad l \in \Zp 
\end{equation*}
defined by
$q_{V^G}(g,u) = u,
 q_{V^G}(g,e) = e$
yield a surjective homomorphism 
of $\lambda$-graph systems,
which we write
$q:{\frak L}^{G,\ell} \rightarrow {\frak L}$.

Let $\eta:E^G\rightarrow G$ 
be the map defined by
$$
\eta(e^g) = g \quad \text{ for } e^g = (g,e) \in E^G_{l,l+1}.
$$
The map $\eta$ satisfies the following conditions:
\begin{gather} 
\eta(e^g) = \eta(f^h)\cdot \ell_G(f)
\quad
\text{ for } 
e^g, f^h \in E_{l,l+1}^G,\, l\in \N 
\text{ with }
t(f^h)=s(e^g), \label{eq:eta2} \\
\eta(e^g) = \eta({e'}^{g'}) \quad
\text{ for } e^g, {e'}^{g'} \in E_{0,1}^G
\text{ with } s(e^g) = s({e'}^{g'}). \label{eq:eta3}
\end{gather}
Conversely, we have the following proposition.
\begin{proposition}
Let ${\frak L} = (V,E,\lambda,\iota)$ 
be a $\lambda$-graph system over $\Sigma$
with an action $\rho$ of $G$
on ${\frak L}$.
Suppose that 
$\Sigma = G \times \Sigma_\circ$
and
there exist maps
$\eta_\circ:E\rightarrow G,$
$r_\circ:E\rightarrow \Sigma_\circ$
and
$\ell_\circ: \Sigma_\circ \rightarrow G$
satisfying the following three conditions: 
\begin{equation}
\lambda(e) = (\eta_\circ(e), r_\circ(e)),
\quad
\eta_\circ(\rho^E_g(e)) = g\eta_\circ(e),
\quad
r_\circ(\rho^E_g(e))=r_\circ(e) \label{eq:ell1}
\end{equation}
for
$g \in G, e \in E$ and 
\begin{gather} 
\eta_\circ(e) = \eta_\circ(f)\cdot \ell_\circ(r_\circ(f))
\quad
\text{ for } 
e, f \in E_{l,l+1},\, l\in \N 
\text{ with }
t(f)=s(e), \label{eq:ell2} \\
\eta_\circ(e) = \eta_\circ(e') \quad
\text{ for } e, e' \in E_{0,1}
\text{ with } s(e) = s(e'). \label{eq:ell3}
\end{gather}
Then there exist a $\lambda$-graph system
${\frak L}_\circ = (V_\circ, E_\circ, \lambda_\circ, \iota_\circ)$ 
over $\Sigma_\circ$
and its $G$-extension
${\frak L}_\circ^{G,\ell_\circ}
$
by $\ell_\circ$
such that 
${\frak L}_\circ^{G,\ell_\circ}
$
is isomorphic to
${\frak L}$
as $G$-$\lambda$-graph systems.
\end{proposition}
\begin{proof}
We first note that 
the condition \eqref{eq:ell2} implies the same condition \eqref{eq:ell3} 
for $e, e' \in E_{l,l+1}$ with $l\in \N$.
In fact, suppose that $e,e' \in E_{l,l+1}$ for $l \in \N$
satisfy $s(e) = s(e')$.
Since $l\ge 1$,
One may take $f \in E_{l-1,l}$ such that
$s(e) = s(e') = t(f)$.
By the condition \eqref{eq:ell2},
we have
\begin{equation} 
\eta_\circ(e) = \eta_\circ(f)\cdot \ell_\circ(r_\circ(f))
        = \eta_\circ(e') \label{eq:etacirc}
\end{equation} 
so that the condition \eqref{eq:ell3} holds 
for $e, e' \in E_{l,l+1}$ with $l\in \N$.

Consider the quotient spaces
of $V, E$ which we denote by
$V_{\circ,l} = V_l/G$
and
$E_{\circ,l,l+1} = E_{l,l+1}/G$.
Denote by
$[v] \in V_{\circ,l}$ 
and
$[e] \in E_{\circ,l,l+1}$
the equivalence class of $v \in V_{l}$
and that of $e \in E_{l,l+1}$,
respectively.
We set
$$
s([e]) = [s(e)],\qquad
t([e]) = [t(e)],\qquad
\lambda_\circ([e]) = r_\circ(e),\qquad
\iota_\circ([v]) = [\iota(v)].
$$
It is easy to see that 
${\frak L}_\circ = (V_\circ, E_\circ, \lambda_\circ, \iota_\circ)$ 
is a $\lambda$-graph system over $\Sigma_\circ$
with the map
$\ell_\circ: \Sigma_\circ \rightarrow G$.

We will define maps
$\jmath^E: E_{\circ, l,l+1} \rightarrow E_{l,l+1}$
and
$\jmath^V: V_{\circ, l} \rightarrow V_{l}$
for each $l \in \Zp$.
Set
\begin{equation}
\jmath^E([e]) = \rho^E_{\eta_\circ(e)^{-1}}(e) 
\quad
\text{ for } [e] \in E_{\circ, l,l+1}.
\label{eq:etaE}
\end{equation}
If $[e'] = [e]$ for some $e' \in E_{l,l+1}$,
one sees 
$e' = \rho^E_g(e)$
for some  $g \in G$.
By \eqref{eq:ell1}, one has 
$\eta_\circ(e') = \eta_\circ(\rho^E_g(e)) = g\eta_\circ(e)$ 
so that
\begin{equation*}
\rho^E_{\eta_\circ(e')^{-1}}(e')
=
\rho^E_{\eta_\circ(e)^{-1}g^{-1}}(\rho^E_g(e))
=\rho^E_{\eta_\circ(e)^{-1}}(e).
\end{equation*} 
This shows that the map $\jmath^E([e])$ defined in \eqref{eq:etaE}
is well-defined.
For
$[v] \in V_{\circ, l}$,
take 
$e \in E_{l,l+1}$ such that $v = s(e)$.
Set
\begin{equation}
\jmath^V([v]) = s(\jmath^E([e]))
 (= s(\rho^E_{\eta_\circ(e)^{-1}}(e))
  =  \rho^V_{\eta_\circ(e)^{-1}}(s(e))). \label{eq:etaV}
\end{equation}
If $[v] = [v'] \in V_{\circ,l}$ 
for some $v' \in V_l$,
one may take  $e' \in E_{l,l+1}$ 
and $g \in G$
such that
$v' = s(e')$
and
$\rho_g^V(v) = v'$ so that
$s(\rho_g^E(e)) = s(e')$.
By \eqref{eq:etacirc}, 
one has 
$\eta_\circ(e') = \eta_\circ(\rho^E_g(e)) = g\eta_\circ(e)$ 
so that we have
\begin{equation*}
\jmath^V([v'])
 = \rho^V_{\eta_\circ(e')^{-1}}(s(e')) 
 = \rho^V_{(g\eta_\circ(e))^{-1}}(s(\rho^E_g(e))) 
 = \rho^V_{\eta_\circ(e)^{-1}}(s(e)) 
=\jmath^V([v]).
\end{equation*} 
This shows that the map $\jmath^V([v])$ defined in \eqref{eq:etaV}
is well-defined.

We next define 
$\xi:{\frak L}_\circ^{G,\ell_\circ} \rightarrow {\frak L}$
by
a pair of bijective maps
$\xi =(\xi^V, \xi^E)$
such that 
$\xi^E: E_{\circ,l,l+1}^G \rightarrow E_{l,l+1}$
and
$\xi^V:V_{\circ,l}^G \rightarrow V_l$
for each $l \in \Zp$
by setting
$$
\xi^E(g,[e]) = \rho^E_g(\jmath^E([e])),
\qquad
\xi^V(g,[v]) = \rho^V_g(\jmath^V([v])).
$$
The map $\xi$ is shown to be compatible to the structure of the 
$\lambda$-graph systems 
${\frak L}_\circ^{G,\ell_\circ}$ and ${\frak L}$
in the following way:
\begin{align*}
s(\xi^E(g,[e]))
&= \rho_g^V ( s(\jmath^E ( [e] ))) \\
&= \rho_g^V(\rho^V_{\eta_\circ(e)^{-1}}(s(e))) \\
&= \rho_g^V(\jmath^V([s(e)])) \\
&= \xi^V((g,[s(e)])) \\
&= \xi^V(s(g,[e])).
\end{align*}
We also see that 
\begin{equation*}
t(\xi^E(g,[e])) 
 = \rho_g^V(t(\jmath^E([e]))) 
 = \rho_g^V(\rho^V_{\eta_\circ(e)^{-1}}(t(e))). 
\end{equation*}
Take $f \in E$ such that 
$t(e) = s(f)$ and hence 
$\eta_\circ(f) = \eta_\circ(e) \ell_\circ(r_\circ(e))$.
It then follows that
\begin{align*}
\rho_g^V(\rho^V_{\eta_\circ(e)^{-1}}(t(e)))
&=
\rho_g^V(\rho^V_{\ell_\circ(r_\circ(e))}(\rho^V_{\eta_\circ(f)^{-1}}(s(f)))) 
  \\
&=\rho_g^V(\rho^V_{\ell_\circ(r_\circ(e))}(\jmath^V([t(e)])))\\
&=\xi^V(g\ell_\circ(r_\circ(e)),[t(e)])) \\ 
&=\xi^V(g\ell_{\circ G}([e]),t([e]))) \\
&= \xi^V(t(g, [e])) 
\end{align*}
so that
\begin{equation*}
s(\xi^E(g,[e])) = \xi^V(s(g,[e])), \qquad
t(\xi^E(g,[e])) = \xi^V(t(g,[e])).
\end{equation*}
Since  the equalities for $[e] \in E_{\circ,l,l+1}$
hold
\begin{equation*}
\lambda(\jmath^E([e]))
 = (\eta_\circ(\rho^E_{\eta_\circ(e)^{-1}}(e)), 
    r_\circ(\rho^E_{\eta_\circ(e)^{-1}}(e))) 
 = (\eta_\circ(e)^{-1} \eta_\circ(e), r_\circ(e))  
 =(1,r_\circ(e)) \in G\times \Sigma_\circ,
\end{equation*}
we have
\begin{equation*}
\lambda(\xi^E(g,[e]))
 = \rho^\Sigma_g(\lambda(\jmath^E([e]))) 
 = (g, \lambda_\circ([e]))  
 = \lambda^G_\circ(g,[e]).
\end{equation*}
Suppose that
$$
\xi^E(g,[e]) = \xi^E(g',[e'])
\quad
\text{  
for some }
e, e' \in E_{l,l+1}
\text{  and }
g,g' \in G.
$$
It then follws that 
\begin{equation}
\rho^E_g(\jmath^E([e]))
=
\rho^E_{g'}(\jmath^E([e']))
\label{eq:rhoegj} 
\end{equation}
and hence
$$
\rho^E_{g \eta_\circ(e)^{-1}}(e) 
=
\rho^E_{g' \eta_\circ(e')^{-1}}(e') 
$$
so that 
$[e] = [e']$.
By  \eqref{eq:rhoegj}, we have
$g=g'$.
Suppose next that
$$
\xi^V(g,[v]) = \xi^V(g',[v'])
\quad
\text{  
for some }
v, v' \in V_l
\text{  and }
g,g' \in G.
$$
Take $e,e' \in E_{l,l+1}$
such that 
$v=s(e), v' = s(e')$.
It then follows that 
\begin{equation*}
\rho^V_g(s(\jmath^E([e])))
=
\rho^V_{g'}(s(\jmath^E([e'])))
\end{equation*}
and hence
$$
\rho^V_{g \eta_\circ(e)^{-1}}(s(e)) 
=
\rho^V_{g' \eta_\circ(e')^{-1}}(s(e')) 
$$
so that 
$[s(e)] = [s(e')]$
and
$g=g'$.
Therefore both $\xi^E$ and $\xi^V$ are injective.
Their surjectivities are easily seen.
Hence the map 
$\xi =(\xi^V, \xi^E)$
gives rise to an isomorphism
$\xi: {\frak L}_\circ^{G,\ell_\circ} \rightarrow {\frak L}$
of $\lambda$-graph systems.
\end{proof}
By \eqref{eq:eta2}, \eqref{eq:eta3}, 
we obtain the following characterization of $G$-extension of $\lambda$-graph systems.
\begin{theorem}\label{thm:extlambda}
Let
${\frak L} =(V, E,\lambda, \iota)$
be a $\lambda$-graph system over $\Sigma$
and $G$ be a finite group.
Then ${\frak L}$ is a $G$-extension of a 
$\lambda$-graph system ${\frak L}_\circ$
over $\Sigma_\circ$ 
if and only if 
there exist a free action $\rho$ of $G$ on ${\frak L}$
and maps 
$\eta_\circ: E \rightarrow G, 
r_\circ: E \rightarrow \Sigma_\circ$
and
$
\ell_\circ: \Sigma_\circ \rightarrow G
$
such that 
$\Sigma = G \times \Sigma_\circ$ and
\begin{equation}
\lambda(e) = (\eta_\circ(e), r_\circ(e)),
\quad
\eta_\circ(\rho^E_g(e)) = g\eta_\circ(e),
\quad
r_\circ(\rho^E_g(e))=r_\circ(e) \label{eq:thell1}
\end{equation}
for
$g \in G, e \in E$ and 
\begin{gather} 
\eta_\circ(e) = \eta_\circ(f)\cdot \ell_\circ(r_\circ(f))
\quad
\text{ for } 
e, f \in E_{l,l+1},\, l\in \N 
\text{ with }
t(f)=s(e), \label{eq:thell2} \\
\eta_\circ(e) = \eta_\circ(e') \quad
\text{ for } e, e' \in E_{0,1}
\text{ with } s(e) = s(e'). \label{eq:thell3}
\end{gather}
\end{theorem}
For a finite directed labeled graph ${\cal G} =(V,E,\lambda)$ over alphabet $\Sigma$
with a labeling map $\lambda:E\rightarrow \Sigma$,
we have a $\lambda$-graph system 
${\frak L}_{\mathcal{G}}$ by setting
$$
V_l = V, \qquad E_{l,l+1} = E, \qquad 
\iota = \id
\qquad
\text{ for }
l \in \Zp.
$$
$G$-extension of a finite directed graph is written in \cite[p. 6]{BoyleSullivan}.
That is easily generalized to  a finite directed labeled graph. 
It is direct to see that $G$-extension of the $\lambda$-graph system 
${\frak L}_{\mathcal{G}}$
is the $\lambda$-graph system of the $G$-extension of ${\mathcal{G}}$.
Then the condition \eqref{eq:thell3} is deduced from \eqref{eq:thell2}.
Hence we have the following corollary.
\begin{corollary}
Let
${\cal G} =(V,E, \lambda)$
be a finite  directed labeled graph over $\Sigma$
with a labeling map $\lambda:E\rightarrow \Sigma$,
and $G$ be a finite group.
Then ${\cal G}$ is a $G$-extension of a 
finite labeled directed graph ${\cal G}_\circ$
over $\Sigma_\circ$ 
if and only if 
there exist a free action $\rho$ of $G$ on ${\cal G}$
and maps  
$\eta_\circ: E \rightarrow G, 
r_\circ: E \rightarrow \Sigma_\circ
$
and
$
\ell_\circ: \Sigma_\circ \rightarrow G 
$
such that 
$\Sigma = G \times \Sigma_\circ$ and
\begin{equation*}
\lambda(e) = (\eta_\circ(e), r_\circ(e)),
\qquad
\eta_\circ(\rho^E_g(e)) = g\eta_\circ(e),
\quad
r_\circ(\rho^E_g(e))=r_\circ(e) \label{eq:ellc1}
\end{equation*}
for
$g \in G, e \in E$ and 
\begin{equation*} 
\eta_\circ(e) = \eta_\circ(f)\cdot \ell_\circ(r_\circ(f))
\quad
\text{ for } 
e, f \in E
\text{ with }
t(f)=s(e). 
\end{equation*}
\end{corollary}
\section{Subshifts presented by extensions of $\lambda$-graph systems}
In this section we study the subshifts presented by extensions of $\lambda$-graph systems.
Let ${\frak L} =(V,E,\lambda,\iota)$
be a $\lambda$-graph system over $\Sigma$.
Suppose that a map $\ell:\Sigma \rightarrow G$
is given. 
Define
$\ell_G: E\rightarrow G$ by
$\ell_G(e) = \ell(\lambda(e)) \in G$ for $e \in E$. 
Let us denote by
$\Lambda_{{\frak L}^{G,\ell}}^+$
the right one-sided subshift defined by the extension ${\frak L}^{G,\ell}$
of $\lambda$-graph system ${\frak L}$ by $G$.
The following lemma is obvious.
\begin{lemma}
For $g_n \in G, e_n \in E_{n-1,n}, n=1,2,\dots$,
the sequence
$(g_n,e_n) \in E_{n-1,n}^G, n \in \N$
gives an admissible path in the $\lambda$-graph system ${\frak L}^{G,\ell}$
if and only if
$t(e_n) = s(e_{n+1}), g_{n+1} = g_n\ell_G(e_n)$ for all $n \in \N$.
\end{lemma}
The above lemma says that 
$(g_n,\lambda(e_n))_{n \in \N} \in \Lambda_{{\frak L}^{G,\ell}}^+$
if and only if 
$t(e_n) = s(e_{n+1}), g_{n+1} = g_n\ell_G(e_n)$ for all $n \in \N$.
\begin{lemma}
The correspondence
$\varphi^+_G:
(g_n,\lambda(e_n))_{n \in \N} \in  \Lambda_{{\frak L}^{G,\ell}}^+
\rightarrow
(g_1,(\lambda(e_n))_{n \in \N}) \in G\times \Lambda_{\frak L}^+
$
gives rise to a homeomorphism between
$ \Lambda_{{\frak L}^{G,\ell}}^+$ and $G\times \Lambda_{\frak L}^+$.
\end{lemma}
\begin{proof}
By the preceding lemma, 
a sequence
$(g_n,e_n) \in E_{n-1,n}^G, n \in \N$
defines an element 
of $\Lambda_{{\frak L}^{G,\ell}}^+$
if and only if
$t(e_n) = s(e_{n+1}), g_{n+1} = g_n\ell_G(e_n)$ for all $n \in \N$.
The condition
$g_{n+1} = g_n\ell_G(e_n), n \in \N$
implies 
$g_{n} = g_1\ell(\lambda(e_1))\cdots\ell(\lambda(e_{n-1}))$
so that 
the sequence
$
(g_1,(\lambda(e_n))_{n \in \N}) \in G\times \Lambda_{\frak L}^+
$
determines 
$(g_n,\lambda(e_n))_{n \in \N} \in  \Lambda_{{\frak L}^G}^+.
$
Hence
$\varphi^+_G:
 \Lambda_{{\frak L}^{G,\ell}}^+ \rightarrow G\times \Lambda_{\frak L}^+$
defines a homeomorphism.
\end{proof}

We define $G$-actions 
$\rho^{\Lambda_{G,\ell}^+}$
on $\Lambda_{{\frak L}^{G,\ell}}^+$
and
$\rho^{G\times \Lambda^+}$
on $G\times \Lambda_{\frak L}^+$
by setting
\begin{equation*}
\rho^{\Lambda_{G,\ell}^+}_g((a_n,x_n)_{n \in \N}) =(ga_n,x_n)_{n \in \N},
\qquad 
\rho^{G\times\Lambda^+}_g((a,(x_n)_{n \in \N})) =(ga,(x_n)_{n \in \N}).
\end{equation*}
Then it is direct to see that 
the $G$-actions commute with the homeomorphism
$\varphi^+_G:
 \Lambda_{{\frak L}^{G,\ell}}^+ \rightarrow G\times \Lambda_{\frak L}^+$.
Let
$\sigma^+_{{\frak L}^{G,\ell}}((g_n,x_n)_{n\in \N})
=(g_{n+1},x_{n+1})_{n\in \N}
$
and
$\sigma^+_{{\frak L}}((x_n)_{n\in \N})=(x_{n+1})_{n\in \N}
$
be the one-sided shifts on $\Lambda_{{\frak L}^{G,\ell}}^+$ 
and $\Lambda_{\frak L}^+$
respectively.
We define a map 
$\tau_\ell^+:\Lambda_{\frak L}^+ \rightarrow G$
by
$\tau_\ell^+(x) = \ell(x_1)$ 
for 
$x = (x_n)_{n \in \N}\in \Lambda_{\frak L}^+.$  
Then we have
\begin{proposition}
$\varphi_G^+ \circ \sigma^+_{{\frak L}^{G,\ell}}\circ
(\varphi_G^+)^{-1}(g,x) = (g\tau_\ell^+(x),\sigma^+_{{\frak L}}(x))
$
for $(g, x) \in G \times \Lambda_{\frak L}^+$.
\end{proposition}
\begin{proof}
For $g \in G, x = (x_n)_{n \in \N}\in \Lambda_{\frak L}^+$,
we set 
$(g_n, x_n)_{n \in \N} = (\varphi_G^+)^{-1}(g,x)$,
so that 
$$
g_2 = g\ell(x_1),\quad
g_3 = g\ell(x_1)\ell(x_2),\quad
\cdots, \quad
g_n = g\ell(x_1)\cdots \ell(x_{n-1}),
$$  
and hence
$
g_2 = g\tau_\ell^+(x).
$
It then follows that 
\begin{equation*}
\varphi_G^+ \circ \sigma^+_{{\frak L}^{G,\ell}}\circ
(\varphi_G^+)^{-1}(g,x)
=   \varphi_G^+ ((g_{n+1}, x_{n+1})_{n \in \N})   
=   (g_2, x_{n+1})_{n \in \N}   
=    (g\tau_\ell^+(x),\sigma^+_{{\frak L}}(x)).
\end{equation*}
\end{proof}
We next consider the two-sided subshifts 
$\Lambda_{\frak L}$ and $\Lambda_{{\frak L}^{G,\ell}}$ 
presented by the $\lambda$-graph systems
${\frak L}$ and 
${\frak L}^{G,\ell}$, respectively.
They are realized from $\Lambda_{\frak L}^+$ 
and
$\Lambda_{{\frak L}^{G,\ell}}^+$ in the following way.
\begin{align*}
\Lambda_{\frak L}
& = \{ (x_n)_{n \in \Z} \in \Sigma^\Z
\mid 
(x_{i+n})_{n \in \Z} \in \Lambda_{\frak L}^+
\text{ for all } i \in \Z\},\\
\intertext{and}
\Lambda_{{\frak L}^{G,\ell}}
& = \{ (g_n,x_n)_{n \in \Z} \in (G\times\Sigma)^\Z
\mid 
(g_{i+n},x_{i+n})_{n \in \Z} \in \Lambda_{{\frak L}^{G,\ell}}^+
\text{ for all } i \in \Z \},
\end{align*}
respectively.
Then for $g \in G$,
the action
$\rho_g:\Lambda_{{\frak L}^{G,\ell}}\rightarrow \Lambda_{{\frak L}^{G,\ell}}$
defined by
$\rho_g((g_n,x_n)_{n \in \Z}) = (g\cdot g_n,x_n)_{n \in \Z}$
satisfies
$\rho_g \circ \sigma_{{\frak L}^{G,\ell}} =\sigma_{{\frak L}^{G,\ell}} \circ\rho_g$
for $g \in G$.
\begin{lemma}
For $g_n \in G, x_n \in \Sigma, n\in \Z$,
the sequence
$(g_n, x_n)_{n \in \Z} \in G\times \Sigma, n \in \Z$
gives rise to an element of
$\Lambda_{{\frak L}^{G,\ell}}$
if and only if
the following two conditions hold:
\item
(1) for $i \in \Z$, there exists a sequence
$e_{n+i}^i \in E_{n-1,n}, n \in \N$ such that 
$t(e_{n+i}^i) = s(e_{n+1+i}^i)$
and
$
x_{n+i} = \lambda(e_{n+i}^i)$
for all $n \in \N$.
\item
(2)
$ g_{n+1} = g_n\ell(x_n)$ for all $n \in \Z$.
\end{lemma}
We define 
$\tau_\ell:\Lambda_{\frak L}
\rightarrow G$
by setting
$\tau_\ell((x_n)_{n\in \Z}) = \ell(x_0)$.
We then see the following lemma.
\begin{lemma}
The map
$
\varphi_G: \Lambda_{{\frak L}^{G,\ell}}\rightarrow G\times \Lambda_{\frak L}
$
defined by
$
\varphi_G((g_n,x_n)_{n \in \Z})
=(g_0,(x_n)_{n \in \Z})
$
for 
$(g_n,x_n)_{n \in \Z}
\in \Lambda_{{\frak L}^{G,\ell}}
$
yields a homeomorphism
between
$
\Lambda_{{\frak L}^{G,\ell}}
$
and
$G\times \Lambda_{\frak L}
$
such that 
\begin{equation*}
\varphi_G\circ\sigma_{{\frak L}^{G,\ell}}\circ (\varphi_G)^{-1}(g_0,x)
= (g_0\tau_\ell(x),\sigma_{\Lambda_{\frak L}})
\quad
\text{ for }
g_0 \in G, x \in \Lambda_{\frak L}.
\end{equation*}
\end{lemma}
\begin{proof}
We first show the injectivity
of 
$
\varphi_G: \Lambda_{{\frak L}^{G,\ell}}\rightarrow G\times \Lambda_{\frak L}.
$
For 
$
(g_n,x_n)_{n \in \Z}
\in \Lambda_{{\frak L}^G},
$
the equalities
$g_n \ell(x_n) = g_{n+1}, n \in \Z$
hold, so that we see
\begin{align*}
g_{n+1} 
&= g_n \ell(x_n)
 = g_0 \ell(x_0)\ell(x_1)\cdots\ell(x_n), \\
g_{-n-1} 
&= g_{-n} \ell(x_{-n})^{-1}
 = g_0 (\ell(x_{-n-1})\ell(x_{-n})\cdots\ell(x_{-1}))^{-1}.
\end{align*}
Hence
$g_n, g_{-n}, n=1,2,\dots $ are determined by
$g_0, x_n, n=0,1,2,\dots$.
Therefore we know that 
$
\varphi_G: \Lambda_{{\frak L}^{G,\ell}}\rightarrow G\times \Lambda_{\frak L}
$
is injective and hence a homeomorphism.

For $g_0 \in G, x = (x_n)_{n\in \N} \in \Lambda_{\frak L},$
put
$(g_n,x_n)_{n \in \Z} 
=\varphi_G^{-1}(g_0, (x_n)_{n \in \Z})$
so that 
$g_n \ell(x_n) = g_{n+1}, n \in \Z$.
It then follows that 
\begin{align*}
\varphi_G\circ\sigma_{{\frak L}^{G,\ell}}\circ\varphi_G^{-1}(g_0,x)
& = \varphi_G((g_{n+1},x_{n+1})_{n \in \Z})\\
& = (g_{1},(x_{n+1})_{n \in \Z})\\
& = (g_{0}\ell(x_0),(\sigma_{\frak L}(x_n))_{n \in \Z}) \\
& = (g_{0}\tau_\ell(x),\sigma_{\frak L}(x))
\end{align*}
so that we have 
$
\varphi_G\circ\sigma_{{\frak L}^{G,\ell}}\circ\varphi_G^{-1}(g_0,x)
 = (g_{0}\tau_\ell(x),\sigma_{\frak L}(x)).
$
\end{proof}
Since the above homeomorphism
$
\varphi_G: \Lambda_{{\frak L}^{G,\ell}}\rightarrow G\times \Lambda_{\frak L}
$
commutes with their $G$-actions,
we see the following result.
\begin{theorem}\label{thm:G-conj}
The subshift 
$ 
(\Lambda_{{\frak L}^{G,\ell}}, \sigma_{{\frak L}^{G,\ell}})
$ 
presented by
the $G$-extension ${\frak L}^{G,\ell}$ 
of $\lambda$-graph system ${\frak L}$ 
by a map $\ell:\Sigma\rightarrow G$
 is $G$-conjugate to the 
skew product 
$(G\ltimes \Lambda_{\frak L}, \tau_\ell\ltimes \sigma_{\frak L})$.
Hence the subshift presented by a $G$-$\lambda$-graph system is a $G$-subshift.
\end{theorem}

\section{$G$-Strong shift equivalence}
In this section,
we first introduce two kinds of notions of $G$-strong shift equivalences 
between \smss over $\Zp G$.
Both of them are generalizations of 
$G$-strong shift equivalence for finite symbolic matrices over $\Zp G$.
They are defined by  analogous ways to the strong shift equivalences between 
symbolic matrices.

Let $\psi:\Sigma_1\rightarrow \Sigma_2$
be a specification from an alphabet $\Sigma_1$ to an alphabet $\Sigma_2$.
Suppose that maps 
$\ell_1:\Sigma_1\rightarrow G$
and
$\ell_2:\Sigma_2\rightarrow G$
are given.
If the equality
$\ell_1 = \ell_2\circ \psi$
holds,
the specification $\psi$ is said to be
compatible to the maps $\ell_1$ and $\ell_2$.
In what follows, we fix a finite group $G$.
Let $(\M,I)$ and
$(\M^{\prime},I')$ be \smss over
$\Sigma$
and
$\Sigma'$, respectively.
Suppose  that 
maps 
$\ell:\Sigma\rightarrow G$
and
$
\ell':\Sigma'\rightarrow G$
are given.
By relabeling 
$\alpha \in \Sigma$
and
$\alpha' \in \Sigma'$
by 
$\ell(\alpha) \in G$
and
$\ell'(\alpha')\in G,$
respectively,
we may regard  the \smss 
$(\M,I)$
and
$(\M^{\prime},I')$ 
as those over 
$\Z_+G$, respectively.
The relabeled \smss
are denoted by
$(\M^{\ell},I)$
and
$(\M^{\prime \ell'},I'),$
respectively.
Each entry 
$
\M^{\ell}_{l,l+1}(i,j)
$
of the matrix
$
\M^{\ell}_{l,l+1}
$
is
\begin{equation*}
\M^{\ell}_{l,l+1}(i,j) =
\ell(\alpha_1) + \cdots +\ell(\alpha_n) 
\quad
\text{ if }
\quad
\M_{l,l+1}(i,j) =
\alpha_1 + \cdots + \alpha_n. 
\end{equation*}
Similarly the matrix 
$
\M^{\prime \ell'}_{l,l+1}
$
is defined.
Let
$m(l)$
and
$m^\prime(l)$
be the sequences for which 
$\M_{l,l+1}, I_{l,l+1}$ are $m(l)\times m(l+1)$ matrices and
$\M'_{l,l+1}, I'_{l,l+1}$ are $m'(l)\times m'(l+1)$ matrices,
 respectively.
\begin{definition}\label{defn:proGSSE}
Symbolic matrix systems
 $(\M^{\ell},I)$ and $(\M^{\prime \ell'},I')$
 over $\Z_+G$
 are said to be {\it properly $G$-strong shift equivalent in 1-step},
written as 
$
(\M^{\ell},I)\underset{G,1-pr}{\approx} (\M^{\prime \ell'},I'),
$
 if
the following two conditions hold:
\begin{enumerate}
\renewcommand{\labelenumi}{(\arabic{enumi})}
\item 
there exist sequences of matrices ${\P}_l, {\Q}_l, X_l, X'_l$ 
 for each $l \in \Zp$ and specifications 
$\varphi:\Sigma\rightarrow C\cdot D,$
$\phi:\Sigma'\rightarrow D\cdot C$
satisfying \eqref{eq:MPQ},\eqref{eq:IXX}, \eqref{eq:XP}, that is 
$
(\M,I)\underset{1-pr}{\approx} (\M',I'),
$
\item 
there exist maps $\ell_C:C\rightarrow G, \ell_D:D\rightarrow G$
such that 
$\varphi:\Sigma\rightarrow C\cdot D$
is compatible to
$\ell$ and $\ell_{CD}$,
where $\ell_{CD}$ is defined by $\ell_{CD}(cd) = \ell_C(c) \ell_D(d), c\in C, d \in D,$
 and similarly 
 $\phi:\Sigma'\rightarrow D\cdot C$
is compatible to
$\ell'$ and $\ell_{DC}$,
and the equalities
\begin{equation}
\M^{\ell}_{l,l+1} 
= {\P}_{2l}^{\ell_C}{\Q}_{2l+1}^{\ell_D}, \qquad
 \M^{\prime \ell'}_{l,l+1} 
= {\Q}_{2l}^{\ell_D}{\P}_{2l+1}^{\ell_C} \label{eq:GMM'}
\end{equation}
hold for $l \in \Zp$.
\end{enumerate} 
 \end{definition}
Symbolic matrix systems 
 $(\M^{\ell},I)$ and $(\M^{\prime \ell'},I')$ 
 over $\Z_+G$
 are said to be {\it properly $G$-strong shift equivalent in N-step},
written as
$
(\M^{\ell},I) \underset{G, N-pr}{\approx} (\M^{\prime \ell'},I'),
$
 if
there exists an $N$-string of properly $G$-strong shift equivalences
in 1-step connecting
between 
$
(\M^{\ell},I)
$
and
$(\M^{\prime \ell'},I').
$
We simply call it a {\it properly $G$-strong shift equivalence}.
 It is straightforward to see that 
properly $G$-strong shift equivalence
is an equivalence relation in 
the set of \smss over $\Z_+G$.

  We will next introduce the notion of $G$-strong shift equivalence 
between two \smss 
that is simpler 
and weaker than properly $G$-strong shift equivalence.
\begin{definition}
Symbolic matrix systems
$(\M^{\ell},I)$
and
$(\M^{\prime \ell'},I')$ over $\Z_+G$ 
 are said to be $G$-{\it strong shift equivalent in 1-step\/},
written as
$
(\M^{\ell},I) \underset{G,1-st}{\approx}(\M^{\prime \ell'},I'),
$
if
 the following two conditions hold:
\begin{enumerate}
\renewcommand{\labelenumi}{(\arabic{enumi})}
\item 
there exist sequences of matrices ${\H}_l, {\K}_l$ 
 for each $l \in \Zp$ and specifications 
$\varphi:\Sigma\rightarrow C\cdot D,$
$\phi:\Sigma'\rightarrow D\cdot C$
satisfying \eqref{eq:IMHK},\eqref{eq:HIIH}, that is 
$
(\M,I)\underset{1-st}{\approx} (\M',I'),
$
\item 
there exist maps $\ell_C:C\rightarrow G, \ell_D:D\rightarrow G$
such that 
$\varphi:\Sigma\rightarrow C\cdot D$
is compatible to
$\ell$ and $\ell_{CD}$, 
 and similarly 
 $\phi:\Sigma'\rightarrow D\cdot C$
is compatible to
$\ell'$ and $\ell_{DC}$,
and the equalities
\begin{equation}
I_{l-1,l}{\M^{\ell}}_{l,l+1} 
= {\H}_l^{\ell_C}{\K}_{l+1}^{\ell_D},  \qquad
I^{\prime}_{l-1,l}{\M}^{\prime \ell'}_{l,l+1} 
= {\K}_l^{\ell_D}{\H}_{l+1}^{\ell_C} \label{eq:GIMHK}
\end{equation}
hold for $l \in \Zp$.
\end{enumerate} 
\end{definition}
Symbolic matrix systems 
 $(\M^{\ell},I)$ and $(\M^{\prime \ell'},I')$ over $\Z_+G$
 are said to be $G$-{\it strong shift equivalent in N-step},
written as
$
(\M^{\ell},I) \underset{G,N-st}{\approx} (\M^{\prime \ell'},I'),
$
 if
there exists an $N$-string of $G$-strong shift equivalences
in 1-step connecting
between 
$
(\M^{\ell},I)
$
and
$(\M^{\prime \ell'},I').
$
 We simply call it a $G$-{\it strong shift equivalence\/}.
Similarly to the case of properly $G$-strong shift equivalence,
$G$-strong shift equivalence in \smss over $\Z_+G$
is  an equivalence relation.
\begin{proposition} 
Properly $G$-strong shift equivalence in 1-step implies
$G$-strong shift equivalence in 1-step.
\end{proposition}
\begin{proof}
Let
$\P_l,\Q_l, X_l$ and $ X'_l$ be the matrices 
in Definition \ref{defn:proGSSE} of properly $G$-strong shift equivalence in 1-step
between 
$(\M^{\ell},I) $
and
$(\M',I').$
We set
$$
\H_l = X_{2l-1}\P_{2l-1},\qquad
\K_l = X'_{2l-1}\Q_{2l-1}
\quad
\text{ for }
l \in \N.
$$
They give rise to a $G$-strong shift equivalence in 1-step between
$(\M^{\ell},I)$ and $(\M^{\prime \ell'},I')$.
\end{proof}
Conversely we know the following proposition.
Since its proof is similar to that of \cite[Theorem 4.3]{ETDS2003},
we omit it. 
\begin{proposition}[{cf. \cite{ETDS2003}}]
Let
$(\M^{\ell},I)$ and 
$(\M^{\prime \ell'},I')$ be the symbolic matrix systems 
over $\Zp G$.
Suppose that their respect $\lambda$-graph systems ${\frak L}$ and ${\frak L}^{\prime}$
for $(\M,I)$ and $(\M', I')$  
are left-resolving and predecessor-separated.
The following are equivalent:
\begin{enumerate}
\renewcommand{\labelenumi}{(\roman{enumi})}
\item 
$(\M^{\ell},I)$ and $(\M^{\prime \ell'},I')$ are properly $G$-strong shift equivalent in $l$-step.     
\item 
$(\M^{\ell},I)$ and $(\M^{\prime \ell'},I')$ are $G$-strong shift equivalent in $l$-step.     
\end{enumerate}
\end{proposition}
Hence the two notions,  
properly $G$-strong shift equivalence
 and
$G$-strong shift equivalence, 
coincide with
each other in the canonical \smss
with a map $\ell:\Sigma\rightarrow G$.

Let
${\frak L}$
and
${\frak L}'$
be $\lambda$-graph systems over 
$\Sigma$
and
$\Sigma'$ with maps
$
\ell:\Sigma\rightarrow G
$
and
$ 
\ell':\Sigma'\rightarrow G,
$
respectively.
Denote by
$(\M^{\ell},I)$
and
$(\M^{\prime \ell'},I')$ 
their respect \smss 
 over $\Zp G$.
Define the skewing functions
$\tau_\ell: \Lambda_{{\frak L}}\rightarrow G
$ and 
$
 \tau_{\ell'}: \Lambda_{{\frak L}'}\rightarrow G
$
 by
\begin{align*}
\tau_\ell((x_n)_{n\in \Z}) 
& = \ell(x_0), \qquad (x_n)_{n\in \Z}\in \Lambda_{{\frak L}},\\
\tau_{\ell'}((y_n)_{n\in \Z}) 
& = {\ell'}(y_0), \qquad (y_n)_{n\in \Z}\in \Lambda_{{\frak L}'}.
\end{align*}
Suppose that
$
(\M^{\ell},I)
$ 
and 
$(\M^{\prime \ell'},I')
$
are properly $G$-strong shift equivalence in $1$-step.
Since
$
(\M,I)
$ 
and 
$(\M',I')
$
are properly strong shift equivalent,
there exists a topological conjugacy 
$\varPhi: \Lambda_{\frak L} \rightarrow \Lambda_{{\frak L}'}$
between their presenting subshifts
coming from the half-shift:
$(c_n d_n)_{n \in \Z} \rightarrow  
(d_n c_{n+1})_{n \in \Z}.
$
It is called the forward bipartite conjugacy
in \cite{Nasu}(cf. \cite{DocMath1999}). 
\begin{lemma}\label{lem:Gst1step}
Suppose that
$
(\M^{\ell},I)
$ 
and 
$(\M^{\prime \ell'},I')
$
are properly $G$-strong shift equivalence in $1$-step.
Then
there exists 
a map
$b:\Sigma\rightarrow G$
such that 
\begin{equation}
\tau_\ell(x)
=\gamma_b(x)\tau_{\ell'}(\varPhi(x))\gamma_b(\sigma_{\Lambda_{\frak L}}(x))^{-1},
\qquad
x \in \Lambda_{\frak L}, \label{eq:cocycle}
\end{equation}
where
$\varPhi: \Lambda_{{\frak L}}\rightarrow \Lambda_{{\frak L}'}$
is the forward bipartite conjugacy
and
$\gamma_b:\Lambda_{{\frak L}}\rightarrow G$
is defined by
$
\gamma_b((x_n)_{n\in \Z}) 
= b(x_0).
$
\end{lemma}
\begin{proof}
Assume that 
\smss $(\M^{\ell},I)$
and
$(\M^{\prime \ell'},I')$  
are properly $G$-strong shift equivalent in 1-step.
Since
$
(\M,I) \underset{1-pr}{\approx} (\M',I'),
$
their presenting  subshifts
$
\Lambda_{\frak L}
$
and
$
\Lambda_{{\frak L}'}
$
are bipartitely related
as in the proof of 
\cite[Theorem 4.1]{DocMath1999}.
 Through the specifications
$\varphi:\Sigma\rightarrow C\cdot D$
 and
$\phi:\Sigma'\rightarrow D\cdot C$,
for any
$(x_n)_{n \in \Z}\in \Lambda_{{\frak L}}$
we may write
$\varphi(x_n) = c_n d_n, n \in \Z$
for some 
$c_n\in C,  d_n \in D, n \in \Z$. 
By putting
$y_n = \phi^{-1}(d_n c_{n+1}), n \in \Z$,
we know that 
$(y_n)_{n \in \Z}$ belongs to 
$\Lambda_{{\frak L}'}$.
The  forward bipartite conjugacy
$\varPhi:\Lambda_{{\frak L}}\rightarrow \Lambda_{{\frak L}'}$
is defined by
$\varPhi((x_n)_{n \in \Z}) = (y_n)_{n \in \Z}.$
We set
$b:\Sigma \rightarrow G$
by
$b(x_0) = \ell_C(c_0)$.
Since
$\tau_\ell((x_n)_{n\in \Z}) 
 = \ell_{CD}(\varphi(x_0)) 
 = \ell_{CD}(c_0 d_0)
$
and
$
\tau_{\ell'}((y_n)_{n\in \Z}) 
= \ell'(y_0) 
= \ell_{DC}(d_0 c_1),
$
by setting 
$\gamma_b((x_n)_{n\in \Z}) 
=b(x_0) \in G$,
 the following equalities hold:
\begin{align*}
\tau_{\ell'}(\varPhi((x_n)_{n\in \Z}))
& =\ell_{DC}(d_0 c_1)\\
& = \ell_C(c_0)^{-1} \ell_C(c_0) \ell_D(d_0)\ell_C(c_1) \\
& = \ell_C(c_0)^{-1} \ell_{CD}(c_0 d_0)\ell_C(c_1) \\
& = \gamma_b((x_n)_{n\in \Z})^{-1} 
\tau_\ell((x_n)_{n\in \Z})
    \gamma_b((x_{n+1})_{n\in \Z}). 
\end{align*}
This shows the equality
\eqref{eq:cocycle}.
\end{proof} 
\begin{proposition}\label{prop:Gskew}
Suppose that
$
(\M^{\ell},I) 
$
and
$
(\M^{\prime \ell'},I')
$
are
properly $G$-strong shift equivalent.
Then
there exist a topological conjugacy
$\varPhi: \Lambda_{{\frak L}}\rightarrow \Lambda_{{\frak L}'}$
and a continuous function
$\gamma:\Lambda_{{\frak L}}\rightarrow G$
such that 
\begin{equation}
\tau_\ell(x)
=\gamma(x)\tau_{\ell'}(\varPhi(x))\gamma(\sigma_{\Lambda_{{\frak L}}}(x))^{-1},
\qquad
x \in \Lambda_{\frak L}. \label{eq:tauellx}
\end{equation}
\end{proposition}
\begin{proof}
Assume that
$
(\M^{\ell},I) 
$
and
$
(\M^{\prime \ell'},I')
$
are properly $G$-strong shift equivalent
in $N$-step.
There exists a finite sequence of \smss
$(\M_{n}, I_{n})$
over $\Sigma_{n}$
and maps
$\ell_{n}:\Sigma_{n}\rightarrow G$
for $n=1,\dots,N$
such that 
\begin{gather*}
(\M,I) = (\M_{1}, I_{1}),\quad
\ell = \ell_{1},\qquad
(\M',I') = (\M_{N}, I_{N}),\quad
\ell'= \ell_{N},\\
(\M^{\ell_n}_{n},I_{n}) \underset{G,1-pr}{\approx} 
( \M^{\ell_{n+1}}_{n+1}, I_{n+1})
\quad \text{ for } n=1,\dots N-1.
\end{gather*}
Let
${\frak L}_{n}$ be the associated $\lambda$-graph system
to
$(\M_{n},I_{n}), n=1,\dots,N$.
By Lemma \ref{lem:Gst1step},
there exist a finite sequence
of  bipartite conjugacies
$\varPhi_n: \Lambda_{{\frak L}_n}\rightarrow \Lambda_{{\frak L}_{n+1}}$
and continuous maps
$\gamma_n:\Lambda_{{\frak L}_n}\rightarrow G$
for $n=1,\dots,N$
such that 
\begin{equation}
\tau_{\ell_n}(x)
=\gamma_n(x)\tau_{\ell_{n+1}}(\varPhi_n(x))\gamma_n(\sigma_{\Lambda_{{\frak L}_n}}(x))^{-1},
\qquad
x \in \Lambda_{{\frak L}_n}.
\end{equation}
Put
\begin{align*}
\gamma(x) & = 
\gamma_1(x)\gamma_2(\varPhi_1(x))\gamma_3(\varPhi_2(\varPhi_1(x)))\cdots
\gamma_N((\varPhi_{N-1}\circ \cdots \circ \varPhi_1)(x)), \\
\varPhi(x) & =
(\varPhi_{N-1}\circ\cdots \circ \varPhi_1)(x), \qquad x \in \Lambda_{{\frak L}_1}.
\end{align*}
We have
\begin{equation*}
\tau_{\ell_1}(x)
=\gamma(x)\tau_{\ell_N}(\varPhi(x))
\gamma (\sigma_{\Lambda_{{\frak L}_1}}(x))^{-1},
\qquad
x \in \Lambda_{{\frak L}_1}.
\end{equation*}
Since $\tau_{\ell} =\tau_{\ell_1}$, $\tau_{\ell'} = \tau_{\ell_N}$, 
the equality \eqref{eq:tauellx} holds.
\end{proof}
To study the converse implication to the above proposition,
we provide some lemmas.
\begin{lemma}
Let $\Lambda$ be 
a subshift over $\Sigma$
and
 maps $\ell:\Sigma\rightarrow G, \ell':\Sigma\rightarrow G$
be given.
Suppose that a continuous map
$\gamma:\Lambda\rightarrow G$
satisfies  
\begin{equation}
\tau_\ell(x) 
 = \gamma(x) \tau_{\ell'}(x) \gamma(\sigma(x))^{-1}, \qquad x \in \Lambda.
\label{eq:taugamma}
\end{equation}
Take $x=(x_n)_{n\in \Z}, y=(y_n)_{n\in \Z}\in \Lambda$. 
\begin{enumerate}
\renewcommand{\labelenumi}{(\roman{enumi})}
\item 
If
$x_n = y_n$ for all $n=0,1,2,\dots$,
we have
$\gamma(x) = \gamma(y)$. 
\item
If
$x_n = y_n$ for all $n=-1,-2,-3,\dots$,
we have
$\gamma(x) = \gamma(y)$. 
\end{enumerate}
\end{lemma}
\begin{proof}
We set 
$\tau(x) = \tau_\ell(x), \tau'(x) = \tau_{\ell'}(x)$.

(i)
We define
$
\tau^n(x) = \tau(x)\tau(\sigma(x))\cdots\tau(\sigma^{n-1}(x))
$
and
$
{\tau'}^n(x) 
$
is similarly defined.
By \eqref{eq:taugamma},
we have
\begin{equation}
\gamma(x) 
 = 
 \tau^n(x) 
 \gamma(\sigma^n(x)) 
 {\tau'}^n(x)^{-1}, \qquad x \in \Lambda.
\label{eq:gamman}
\end{equation}
Since both
$\tau(x)$ and 
$\tau'(x)$
are determined at the $0$th coordinate $x_0$
of $x$,
the condition
$x_n = y_n, n \in \Zp$
ensures us
$\tau^n(x) = \tau^n(y)$
for all $n \in \N$
and similarly
${\tau'}^n(x) = {\tau'}^n(y)$
for all $n \in \Zp$.
Now $\gamma:\Lambda\rightarrow G$ is continuous,
$\gamma(x)$ 
is determined by finite coordinates of $x$
so that 
one may take $L \in \N$ such that 
$\gamma(\sigma^L(x)) = \gamma(\sigma^L(y))$,
because $x_n = y_n$ for all $n\in \Zp$. 
Hence
$
\gamma(x) = \tau^L(x) \gamma(\sigma^L(x)){\tau'}^L(x)^{-1}.
$
Since
$\tau^L(x) = \tau^L(y)$
and
${\tau'}^L(x) = {\tau'}^L(y)$, 
we get
$\gamma(x) = \gamma(y)$. 

(ii)
By \eqref{eq:taugamma}, we have
\begin{equation}
\gamma(x) 
 = \tau(\sigma^{-1}(x))^{-1}\gamma(\sigma^{-1}(x))
   \tau'(\sigma^{-1}(x))^{-1}, \qquad x \in \Lambda.
\label{eq:gammainverse}
\end{equation}
We define
$
\tau^{-n}(x) 
= \tau(\sigma^{-n}(x))\cdots\tau(\sigma^{-2}(x))\tau(\sigma^{-1}(x))
$
and
$
{\tau'}^{-n}(x)
$
is similarly defined,
so that we have
\begin{equation}
\gamma(x) 
 = 
 \tau^{-n}(x)^{-1} 
 \gamma(\sigma^{-n}(x)) 
 {\tau'}^{-n}(x), \qquad x \in \Lambda.
\label{eq:gammaninverse}
\end{equation}
Since
$x_n = y_n$ for all $n=-1, -2,\dots$,
we have
$\tau^{-n}(x) = \tau^{-n}(y)$
for all $n \in \N$
and 
${\tau'}^{-n}(x) = {\tau'}^{-n}(y)$
for all $n \in \Zp$.
Similarly to the above discussion,
one may take $L \in \N$ such that 
$\gamma(\sigma^{-L}(x)) = \gamma(\sigma^{-L}(y))$
and hence
$
\gamma(x) = \tau^{-L}(x)^{-1} \gamma(\sigma^{-L}(x)){\tau'}^L(x).
$
Since
$\tau^{-L}(x) = \tau^{-L}(y)$
and
${\tau'}^{-L}(x) = {\tau'}^{-L}(y)$, 
we get
$\gamma(x) = \gamma(y)$. 
\end{proof}
For $x = (x_n)_{n \in \Z} \in \Lambda$,
we set
$x^+ =(x_n)_{n \in \Zp}\in \Lambda^+.$
For $K\le L$,
let us denote by 
$x_{[K,L]}$ 
the finite word $(x_K,x_{K+1},\dots, x_L)\in B_{L-K+1}(\Lambda)$
for $x =(x_n)_{n \in \Z}$.
\begin{lemma}
Let $\Lambda$ be 
a subshift over $\Sigma$
and maps $\ell:\Sigma\rightarrow G, \ell':\Sigma\rightarrow G$
be given.
Suppose that a continuous map
$\gamma:\Lambda\rightarrow G$
satisfies the equality \eqref{eq:taugamma}.  
Then there exists $L \in \N$ such thst 
for 
$x, y \in \Lambda$
and
 $l \ge L$,
if $x^+ \sim_{l} y^+$ $l$-past equivalent,
then $\gamma(x) = \gamma(y)$.
\end{lemma}
\begin{proof}
We set 
$\tau(x) = \tau_\ell(x), \tau'(x) = \tau_{\ell'}(x)$.
Since $\gamma:\Lambda \rightarrow G$ is continuous,
$\gamma$ is determined by the window $[-L_1, L_1]$ 
between $-L_1$th coordinate and $L_1$th coordinate.
Hence 
the equality
$\gamma(x) = \gamma(y)$ holds
 for $x,y \in \Lambda$
with $x_{[-L_1,L_1]} = y_{[-L_1,L_1]}$.
Put $L = 2 L_1 +1$.
Suppose that 
$x_{[-L, -1]} = y_{[-L, -1]}$. 
We have
$$
\tau^{-L_1-1}(x) =\tau^{-L_1-1}(y),
\qquad
{\tau'}^{-L_1-1}(x) ={\tau'}^{-L_1-1}(y)
$$
and
$$
\sigma^{-L_1-1}(x)_{[-L_1, L_1]} 
=x_{[-L, -1]} 
=y_{[-L, -1]} 
=\sigma^{-L_1-1}(y)_{[-L_1, L_1]}.
$$
By \eqref{eq:gammaninverse}, we get $\gamma(x) = \gamma(y)$.
For $l \ge L$,
suppose that $x^+ \sim_{l} y^+$ $l$-past equivalent.
One may find a word $\mu \in B_L(\Lambda)$
and
$\tilde{x}, \tilde{y} \in \Lambda$ 
such that
$\tilde{x}_{[-L,-1]} =\mu, \tilde{x}^+ = x^+$
and
$\tilde{y}_{[-L,-1]} =\mu, \tilde{y}^+ = y^+$.
Hence the equality
$\gamma(\tilde{x}) = \gamma(\tilde{y})$ holds.
Since 
$\tilde{x}_n = x_n$ for all $n=0,1,2,\dots $
we have
$\gamma(\tilde{x}) = \gamma(x)$
and similarly
$\gamma(\tilde{y}) = \gamma(y)$,
so that
$\gamma(x) = \gamma(y)$.
\end{proof}
By the above lemma, 
$\gamma: \Lambda\rightarrow G$ 
defines a sequence of maps
$\gamma_l:V^\Lambda_l \rightarrow G$
from the $l$-past equivalence classes of $\Lambda$ to $G$ for 
$l \ge L$.
Let us denote by
$[x^+]_l$ the $l$-past equivalence class of 
$x^+$.
 We may write
 $\gamma(x) = \gamma([x^+]_l)$. 
\begin{lemma}\label{lem:tauGstone}
Let $(\M,I)$ 
be the canonical symbolic matrix system for a subshift $(\Lambda,\sigma)$
over $\Sigma$.
Let maps $\ell:\Sigma\rightarrow G, \ell':\Sigma\rightarrow G$
be given.
If there exists a map
$b:\Sigma\rightarrow G$
satisfying  
\begin{equation}
\tau_\ell(x) 
 = \gamma_b(x) \tau_{\ell'}(x) \gamma_b(\sigma(x))^{-1}, 
 \qquad x \in \Lambda, \label{eq:tauellgammab}
\end{equation}
then their respect \smss 
$(\M^{\ell},I)$ and $(\M^{\ell'},I)$
are properly $G$-strong shift equivalent in $1$-step,
where
$\gamma_b:\Lambda\rightarrow G$ is defined by
$\gamma_b((x_n)_{n\in \Z})= b(x_0)$.
\end{lemma}
\begin{proof}
Let us denote by 
$
{\frak L}^\Lambda 
=(V^\Lambda, E^\Lambda,\lambda^\Lambda,\iota^\Lambda)
$
the canonical $\lambda$-graph system for $\Lambda$.  
As the continuous map
$\gamma_b:\Lambda\rightarrow G$ is determined by the window
$[-L_1, L_1]$ for $L_1 =0$, one may take $L$ in the previous lemma,
as $L = 2\cdot 0 +1 =1$.
Let $v_{i(x)}^l$ denotes the vertex  $[x^+]_l$ in $V_l^\Lambda$
determined by the $l$-past equivalence class of $x^+$.
 Since $\gamma(x) = \gamma_l([x^+]_l)$,  
 the equality \eqref{eq:tauellgammab}
implies that 
\begin{equation}
\ell(x_0) = \gamma([x^+]_l) \ell'(x_0) \gamma([\sigma(x)^+]_{l+1})^{-1}, 
\qquad
x \in \Lambda. \label{eq:ellgamma}
\end{equation}
We set the $m(l)\times m(l)$ diagonal matrix
$D_l=[D_l(i,i)]_{i=1}^{m(l)}$ with its diagonal entries in $G$ 
for which the $(i(x),i(x))$-component is 
$\gamma([x]_l)$.
The definition of 
$D_l(i,i) $ is well-defined by the preceding lemma.
For $x=(x_n)_{n\in \Z} \in \Lambda$,
the right one-sided sequence
$x^+ = (x_0,x_1,\dots ) \in \Lambda^+$
defines
an edge $e_{i,j} \in E_{l,l+1}^\Lambda$
such that
\begin{equation*}
s(e_{ij}) = v_i^l =[x^+]_l, \qquad
t(e_{ij}) = v_j^{l+1} =[\sigma(x)^+]_{l+1}, \qquad
\lambda^\Lambda(e_{ij}) = x_0,
\end{equation*}
because
$x_0 F_j^{l+1} \subset F_i^l$,
where
$F_i^l = [x^+]_l, F_j^{l+1} = [\sigma(x)^+]_{l+1}.
$
The equality \eqref{eq:ellgamma} means 
\begin{equation}
\ell(\lambda(e_{i,j})) 
= D_l(i,i) \ell'(\lambda(e_{i,j})) D_{l+1}(j,j)^{-1} \label{eq:elleij} 
\end{equation}
which is also written
\begin{equation}
\ell(x_0) 
= b(x_0) \ell'(x_0)) b(x_1)^{-1}. \label{eq:ellxzero} 
\end{equation}
The equality \eqref{eq:elleij}
implies 
\begin{equation}
\M^{\ell}_{l,l+1}(i,j)
= D_l(i,i) \M^{\ell'}_{l,l+1}(i,j) D_{l+1}(j,j)^{-1}
\end{equation} 
so that
\begin{equation}
\M^{\ell}_{l,l+1}
= D_l \M^{\ell'}_{l,l+1} D_{l+1}^{-1}. \label{eq:MDMD}
\end{equation} 
Let $C = G\Sigma =\{g \alpha\mid g \in G, \alpha \in \Sigma\}$
and
$D = G$.
Define specifications
$\varphi:\Sigma\rightarrow C\cdot D$
and
$\phi:\Sigma \rightarrow D\cdot C$ 
by setting
$\varphi(x_0) = b(x_0)x_0\cdot b(x_1)^{-1}$
and
$\phi(x_0) = b(x_0)^{-1}\cdot b(x_0) x_0$.
By \eqref{eq:ellxzero},
$b(x_1)^{-1}$ is determined by $x_0$
so that  $\varphi$ is well-defined.
We define symbolic matrices
$\P_k, \Q_k $ by
\begin{align*}
\P_k & = D_l \M_{l,l+1} \quad \text{ if } k = 2l, 2l+1,\\
\Q_k & =
{
\begin{cases}
D_l^{-1} & \text{ if } k = 2l,\\
D_{l+1}^{-1} & \text{ if } k = 2l+1
\end{cases}
}
\end{align*}
and
matrices 
$X_k, X'_k $ by
\begin{align*}
X_k & =
{
\begin{cases}
E_l & \text{ if } k = 2l,\\
I_{l,l+1} & \text{ if } k = 2l+1,
\end{cases}
}\\
X'_k & =
{
\begin{cases}
I_{l,l+1} & \text{ if } k = 2l,\\
E_l & \text{ if } k = 2l+1,
\end{cases}
}
\end{align*}
where $E_l$ denotes the $m(l) \times m(l)$ identity matrix.
The matrices
$\P_k, \Q_k, X_k, X'_k $
satisfy the conditions
\eqref{eq:MPQ}, \eqref{eq:IXX}, \eqref{eq:XP}
for the symbolic matrix system
$(\M,I) = (\M', I')$.
We set 
$\ell_C:C\rightarrow G,$
$\ell_D:D\rightarrow G$
by
$$ 
\ell_C(b(x_0) x_0) =b(x_0) \ell'(x_0), \qquad
\ell_D = \id.
$$
The equality \eqref{eq:ellxzero}
shows that
\begin{align*}
\ell(x_0)& = \ell_C(b(x_0)x_0) \ell_D(b(x_1)^{-1})
          = \ell_{CD}(\varphi(x_0)), \\
\ell'(x_0)& = b(x_0)^{-1}\cdot b(x_0) \ell'(x_0)          
           =\ell_{DC}(\phi(x_0))
\end{align*}
so that 
$\varphi$ is compatible to $\ell$ and $\ell_{CD}$
and
$\phi$ is compatible to $\ell'$ and $\ell_{DC}.$
The equality
\eqref{eq:MDMD}
implies
\begin{equation*}
\M^{\ell}_{l,l+1} 
= {\P}_{2l}^{\ell_C}{\Q}_{2l+1}^{\ell_D}, \qquad
 \M^{\ell'}_{l,l+1} 
= {\Q}_{2l}^{\ell_D}{\P}_{2l+1}^{\ell_C} 
\end{equation*}
hold for $l \in \Zp$.
Therefore the
\smss $(\M^{\ell},I)$ and $ (\M^{\ell'},I)$
are properly $G$-strong shift equivalent in $1$-step.
\end{proof}
In  \cite[Theorem 7.1]{ETDS2003},
it has been shown that a topological conjugacy 
between two subshifts
is decomposed into a finite chain of 
properly strong shift equivalences 
on their canonical symbolic matrix systems.
Each of the properly strong shift equivalences comes from 
one of the four operations,
merged in-splittings,
merged out-splittings,
merged in-amalgamations
and
merged out-amalgamations
on their canonical symbolic matrix systems.
We may show that 
the four operations,
in-splittings,
out-splittings,
in-amalgamations
and
out-amalgamations,
 induce  
properly $G$-strong shift equivalences.
In the following lemma,
we will prove it in the case of 
in-splittings. 
The other three cases are similarly shown.
Let $(\M,I)$ be the \sms for a $\lambda$-graph system ${\frak L}$
over $\Sigma$.
For a partition ${\mathcal{P}}$
of $\Sigma$,
we have 
the in-splitting  $\lambda$-graph system 
${\frak L}_{[{\mathcal{P}}]}$ over
alphabet
$\Sigma_{\mathcal{P}} = {\mathcal{P}}\times \Sigma$
of ${\frak L}$
by ${\mathcal{P}}$
(\cite[Section 5]{ETDS2003}).
Let us denote by
$(\M_{\mathcal{P}}, I_{\mathcal{P}})$
the associated symbolic matrix systems.
\begin{lemma}
Keep the above situation.
Suppose that  a  map $\ell:\Sigma \rightarrow G$
is given.
\begin{enumerate}
\renewcommand{\labelenumi}{(\roman{enumi})}
\item 
There exists a map
$\ell_{\mathcal{P}}:\Sigma_{\mathcal{P}} \rightarrow G$
such that
the \smss 
$(\M^{\ell}, I)$
and
$(\M_{\mathcal{P}}^{\ell_{\mathcal{P}}}, I_{\mathcal{P}})$
are properly $G$-strong shift equivalent.
\item
There exists a map
$b:\Sigma \rightarrow G$ 
such that 
\begin{equation}
\tau_\ell(x)
=\gamma_b(x)
\tau_{\ell_{\mathcal{P}}}(\varPhi(x))\gamma_b(\sigma_{\Lambda_{\frak L}}(x))^{-1},
\qquad
x \in \Lambda_{\frak L} \label{eq:lPcocycle}
\end{equation}
where 
$\varPhi:\Lambda_{\frak L}\rightarrow \Lambda_{{\frak L}_{[\P]}}$
is the forward bipartite conjugacy arising from the in-splitting
of ${\frak L}$.
\end{enumerate}
\end{lemma}
\begin{proof}
(i)
We put
$C = \Sigma, D = {\mathcal{P}}$.
Let us denote by $[\alpha]$ 
the partition class of a symbol
$\alpha \in \Sigma$ in ${\mathcal{P}}$.
Let us define the specification
$\varphi:\Sigma \rightarrow C\cdot D$ and
$\phi:\Sigma_{\mathcal{P}}\rightarrow D\cdot C$
by
$\varphi(\alpha) = \alpha\cdot[\alpha]$
and
$\phi(p,\alpha) = (p,\alpha)$.
The specifications 
$\varphi,\phi$
yield the properly strong shift equivalence between
$(\M,I)$ and $(\M_{\mathcal{P}}, I_{\mathcal{P}})$
(\cite{ETDS2003}).
We set 
$\ell_C:C\rightarrow G$
and
$\ell_D:D\rightarrow G$
by $\ell_C = \ell$
and
$\ell_D(p) = 1$ the unit of $G$ 
for all $p \in {\mathcal{P}}$.
We define
$\ell_{\mathcal{P}}:\Sigma_{\mathcal{P}} \rightarrow G$
by
$\ell_{\mathcal{P}}(p,c) = \ell(c)$.
We then have for $c \in C$ and $p \in {\mathcal{P}}$
\begin{align*}
\ell_{CD}\circ\varphi(c) 
& = \ell_{CD}(c[c]) 
  = \ell_{C}(c)\ell_{D}([c]) 
  = \ell_{C}(c) 
  = \ell(c), \\
\ell_{DC}\circ\phi(p,c) 
& = \ell_{DC}(p\cdot c)
  =  \ell_{D}(p) \ell_{C}( c)
  =  1\cdot \ell_{C}( c)
  =  \ell_{\mathcal{P}}(p,c).
\end{align*}
Hence 
$\varphi$ is compatible to $\ell$ and $\ell_{CD}$,
and
$\phi$ is compatible to $\ell_{\mathcal{P}}$ and $\ell_{DC}$.
As in \cite[p. 1568]{ETDS2003},
the  $\lambda$-graph systems
${\frak L}$
and
${\frak L}_{[\mathcal{P}]}$
are bipartitely related,
there exist  symbolic matrices
$\P_l$ over $C$ and 
$\Q_l$ over $D$ 
which give rise to properly strong shift equivalence between
$(\M,I)$ and $(\M_{\mathcal{P}}, I_{\mathcal{P}})$.
Since
\begin{equation*}
{\M}_{l,l+1} 
\overset{\varphi}{\simeq} {\P}_{2l}{\Q}_{2l+1}, 
\qquad
{\M}_{[{\mathcal{P}}],l,l+1} 
\overset{\phi}{\simeq} {\Q}_{2l}{\P}_{2l+1}
\end{equation*}
and
$\ell_{CD}\circ\varphi = \ell,
\ell_{DC}\circ\phi = \ell_{[\mathcal{P}]}$
we have
\begin{equation*}
{\M}_{l,l+1}^\ell 
= {\P}_{2l}^{\ell_C}{\Q}_{2l+1}^{\ell_D}, 
\qquad
{\M}_{[\mathcal{P}],l,l+1}^{\ell_{[\mathcal{P}]}} 
={\Q}_{2l}^{\ell_D}{\P}_{2l+1}^{\ell_C},
\end{equation*}
so that 
$(\M^\ell,I)$ and 
$(\M_{\mathcal{P}}^{\ell_{[\mathcal{P}]}}, I_{\mathcal{P}})$
are properly $G$-strong shift equivalent.

(ii)
By Lemma \ref{lem:Gst1step},
there exists the desired map
$b:\Sigma\rightarrow G$
satisfying the equality
\eqref{eq:lPcocycle}. 
\end{proof}
We then have the following proposition.
\begin{proposition} \label{eq:propmain}
Let $(\M,I)$
and
$(\M',I')$ be 
the canonical symbolic matrix systems for subshifts $\Lambda$
and
$\Lambda'$ respectively.
Let maps  $\ell:\Sigma\rightarrow G, \ell':\Sigma'\rightarrow G$
be given.
Suppose that there exist a topological conjugacy
$\varPhi: \Lambda\rightarrow \Lambda'$
and a continuous map
$\gamma:\Lambda\rightarrow G$ such that 
\begin{equation}
\tau_\ell(x) 
 = \gamma(x) \tau_{\ell'}(\varPhi(x)) \gamma(\sigma(x))^{-1}, 
 \qquad x \in \Lambda.
\label{eq:propcocycle}
\end{equation}
Then their respect \smss $(\M^{\ell},I), (\M^{\prime \ell'},I')$
are properly $G$-strong shift equivalent.
\end{proposition}
\begin{proof}
Since
$\varPhi: \Lambda\rightarrow \Lambda'$
is a topological conjugacy,
one knows that the canonical \smss
 $(\M,I), (\M',I')$
are properly strong shift equivalent.
By  \cite[Theorem 7.1]{ETDS2003},
properly strong shift equivalence between
$(\M,I)$ and $(\M',I')$
is given by a finite chain of the four operations,
merged out-splittings,
merged in-splittings,
merged out-amalgamations
and
merged in-amalgamations.
This means that there exists a finite sequence 
$\Lambda_0, \Lambda_1,\dots,\Lambda_N$ 
of subshifts such that
$\Lambda_0 = \Lambda$ and
$\Lambda_N = \Lambda'$
and
their canonical $\lambda$-graph systems
${\frak L}^{\Lambda_i}$ and
${\frak L}^{\Lambda_{i+1}}$
are related by one of the four operations.
Denote by 
$(\M_i, I_i)$ the associated \sms for 
${\frak L}^{\Lambda_i}$. 
Suppose that $(\M_{i+1},I_{i+1})$ 
is obtained 
from $(\M_i,I_i)$
by an in-splitting on their  $\lambda$-graph systems.
By the previous lemma,
for a given
$\ell_i: \Sigma_i \rightarrow G$,
there exist maps
$\ell_{i+1}: \Sigma_{i+1} \rightarrow G$
and
$b_i: \Sigma_i \rightarrow G$
such that 
\begin{equation}
\tau_{\ell_i}(x) 
 = \gamma_{b_i}(x) 
 \tau_{\ell_{i+1}}(\varPhi_i(x)) 
  \gamma_{b_i}(\sigma_i(x))^{-1}, \qquad x \in \Lambda_i,
\label{eq:tauli}
\end{equation}
for $i=0,1,\dots,N-1$,
where
$\varPhi_i:\Lambda_i \rightarrow \Lambda_{i+1}$
are the associated  bipartite conjugacies to the four operations.
Since the merged operation defined in \cite{ETDS2003}
does not change the presenting subshifts,
one knows that 
if 
$(\M_{i+1},I_{i+1})$ 
is obtained 
from $(\M_i,I_i)$
by an merged in-splitting on their  $\lambda$-graph systems,
the same equality as 
\eqref{eq:tauli}
holds.
Similar arguments work for the other 
three operations,
merged out-splittings, merged in-amalgamations and merged out-amalgamations.
Since the given topological conjugacy
$\varPhi:\Lambda \rightarrow \Lambda'$
is the compositions of the bipartite conjugacies
$\varPhi_i, i=1,\dots,N$,
 by starting a given map $\ell:\Sigma\rightarrow G$
 denoted by $\ell_0$
and the subshift $\Lambda = \Lambda_0$,
we finally obtain a map
$\ell_N:\Sigma_N = \Sigma' \rightarrow G$
and a continuous map
$\gamma_N: \Lambda \rightarrow G$ 
such that 
\begin{equation}
\tau_\ell(x) 
 = \gamma_{N}(x) 
 \tau_{\ell_{N}}(\varPhi(x)) 
  \gamma_{N}(\sigma(x))^{-1}, \qquad x \in \Lambda.
\label{eq:tauN}
\end{equation}
By combining the given identity 
\eqref{eq:propcocycle},
we have
\begin{equation}
\tau_{\ell'}(\varPhi(x)) 
 = \gamma(x)^{-1}\gamma_{N}(x)
 \tau_{\ell_{N}}(\varPhi(x)) 
  \gamma_N(\sigma(x))^{-1}\gamma(\sigma(x)), \qquad x \in \Lambda.
\label{eq:tauprime}
\end{equation}
By putting
$\gamma'(y) = \gamma(\varPhi^{-1}(y))^{-1}\gamma_{N}(\varPhi^{-1}(y)),$
we get
\begin{equation}
\tau_{\ell'}(y) 
 = \gamma'(y)
 \tau_{\ell_{N}}(y) 
  \gamma'(\sigma'(y))^{-1}, \qquad y \in \Lambda'.
\label{eq:tauy}
\end{equation}
Since the operations of higher block systems are obtained by in-splittings
or out-splittings
of $\lambda$-graph systems,
one may assume that 
the above map
$ \gamma'$ is a one-block map.
 By Lemma \ref{lem:tauGstone},
 we see that 
$(\M^{\prime \ell'}, I')$ and
$(\M^{\prime \ell_N}, I')$ 
are properly $G$-strong shift equivalent.
Since
$
(\M_i^{\ell_i},I_i)\underset{G,1-pr}{\approx} 
(\M_{i+1}^{\ell_{i+1}},I_{i+1})
$
 for 
 $i=0,1,\dots, N-1$
 and
$
(\M^{\ell},I) =(\M_0^{\ell_0},I_0), 
(\M_{N}^{\ell_N},I_N)=(\M^{\prime \ell_N},I')
$
 we conclude
that
$
(\M^{\ell},I)\underset{G,N-pr}{\approx} 
(\M^{\prime \ell'},I').
$
\end{proof}
Therefore we have the following theorem:
\begin{theorem}\label{thm:main}
Let $G$ be a finite group.
Let ${\frak L}$ and ${\frak L}'$
be $\lambda$-graph systems over $\Sigma$
and
$\Sigma'$
respectively.
Suppose that maps
$\ell:\Sigma\rightarrow G$ and 
$\ell':\Sigma'\rightarrow G$
are given.
Let $(\M^{\ell},I)$ and $(\M^{\prime \ell'},I')$ be their \smss over $\Zp G$ 
through the maps $\ell$ and $\ell',$ respectively.
Consider the following three conditions:
\begin{enumerate}
\renewcommand{\labelenumi}{(\arabic{enumi})}
\item 
$(\M^{\ell},I)$ and $(\M^{\prime \ell'},I')$ are properly $G$-strong shift equivalent.     
\item 
There exists a topological conjugacy
$\varPhi:\Lambda_{\frak L}\rightarrow\Lambda_{{\frak L}'}$
such that $\tau_\ell$ is cohomologous to $\tau_{\ell'}\circ\varPhi$.
\item 
The $G$-subshifts $\Lambda_{{\frak L}^{G,\ell}}$
and
$\Lambda_{{{\frak L}'}^{G,\ell'}}$
are $G$-conjugate.
\end{enumerate}
Then we have
$$
(1) \Longrightarrow (2) \Longleftrightarrow (3).
$$
If in particular,
${\frak L}$ and ${\frak L}'$ are both the canonical $\lambda$-graph systems,
we have
$
(2) \Longrightarrow (1).
$
\end{theorem}
\begin{proof}
The implications
$(2) \Longleftrightarrow (3)$
come from Proposition \ref{prop:G-actioncoho}.
The implication
$
(1) \Longrightarrow (2)
$
comes from Proposition \ref{prop:Gskew}.
In case that both 
${\frak L}$ and ${\frak L}'$
are the canonical $\lambda$-graph systems,
the associated symbolic matrix systems 
$(\M,I)$ and $(\M',I')$
are the canonical symbolic matrix systems
for
$
\Lambda_{\frak L}
$
and
$
\Lambda_{{\frak L}'}
$
respectively.
Hence the implication 
$
(2) \Longrightarrow (1)
$
in this case
follows from Proposition \ref{eq:propmain}.
\end{proof}

\section{ Example and Markov--Dyck shifts}
In the first half of this section,
we present an extension of 
a certain $\lambda$-graph system
presenting Dyck shift $D_2$
by a finite group.
In the second half of this section,
we study $G$-conjugacy classes of extensions of a family of nonsofic subshifts called Markov--Dyck shifts.

{\bf 1.}
The  Dyck shift $D_2$ is a typical example of a non-sofic subshift
(cf. \cite{Kr}, \cite{KM2}, \cite{JOT2007}).
We consider the Dyck shift $D_2$ with alphabet 
$\Sigma = \Sigma^- \cup \Sigma^+$
where
$\Sigma^- = \{ \alpha_1,\alpha_2 \},
\Sigma^+ = \{ \beta_1,\beta_2 \}.
$
The symbols 
$ \alpha_i, \beta_i$
correspond to 
the brackets
$(_i,  )_i $
respectively.
To define the Dyck shift,
we use the Dyck inverse monoid $\mathbb{D}_2$
which is a monoid having the relations
\begin{equation}
\alpha_i \beta_j
=
\begin{cases}
 {\bold 1} & \text{ if } i=j,\\
 0  & \text{ otherwise} 
\end{cases} 
\label{eq:D2}
\end{equation}
for $ i,j = 1,2$ (\cite{Kr}, \cite{KM2}).
A word 
$\omega_1\cdots\omega_n $ 
of $\Sigma$
is admissible for $D_2$ precisely if
$
\prod_{m=1}^{n} \omega_m \ne 0.
$
For a word $\omega= \omega_1 \cdots \omega_n $ of $\Sigma,$ 
we denote by $\tilde{\omega}$ its reduced form.
Namely $\tilde{\omega}$ is a word of $\Sigma \cup \{ 0, {\bold 1} \}$
obtained after the operations \eqref{eq:D2}.
Hence a word $\omega$ of $\Sigma$
is forbidden for $D_2$ if 
and only if $\tilde{\omega} = 0$.
There are two typical $\lambda$-graph systems presenting $D_2$.
One is the canonical $\lambda$-graph system 
${\frak L}^{D_2}$,
and the other one is so called the Cantor horizon $\lambda$-graph system
written $\LCHD2$.
The former one corresponds to
the left Krieger cover and
the latter one does to
the left Fischer cover.
Although the former ${\frak L}^{D_2}$
is not irreducible
as a $\lambda$-graph system,
the latter $\LCHD2$
is irreducible
so that the associated $C^*$-algebra 
${\mathcal{O}}_{\LCHD2}$
is simple and purely infinite (\cite{DocMath2002}).
In this section, we treat the latter one $\LCHD2$.

Let us describe the Cantor horizon $\lambda$-graph system $\LCHD2$ of $D_2$
which has been introduced in \cite{KM2} (cf. \cite{JOT2007}).
Let $\Sigma_2$ be the full $2$-shift 
$\{ 1, 2 \}^{\Z}$.
We denote by 
$B_l(D_2)$
and 
$B_l(\Sigma_2)$
 the set of admissible words of length 
$l$ of $D_2$
and that of 
$\Sigma_2$ respectively.
The vertices $V_l$ of $\LCHD2$
 at level $l$
are given by the words of length $l$
consisting of the symbols of $\Sigma^+$
such as 
$$
V_l = \{(\beta_{\mu_1}\cdots\beta_{\mu_l}) \in B_l(D_2)
 \mid \mu_1\cdots\mu_l\in B_l(\Sigma_2) \}.
$$
Hence the cardinal number of $V_l$ is $2^l$.
The mapping $\iota ( = \iota_{l,l+1}) :V_{l+1}\rightarrow V_l$ 
is defined by deleting the rightmost symbol of a word such as 
$$ 
\iota(\beta_{\mu_1}\cdots\beta_{\mu_{l+1}}) 
= (\beta_{\mu_1}\cdots\beta_{\mu_l}),
\qquad (\beta_{\mu_1}\cdots\beta_{\mu_{l+1}}) \in V_{l+1}.
$$
There exists an edge labeled $\alpha_j$ from
$(\beta_{\mu_1}\cdots\beta_{\mu_l})\in V_l$ 
to
$(\beta_{\mu_0}\beta_{\mu_1}\cdots\beta_{\mu_l})\in V_{l+1}$
precisely if
$\mu_0 = j,$
and 
there exists an edge labeled $\beta_j$ from
$(\beta_j\beta_{\mu_1}\cdots\beta_{\mu_{l-1}})\in V_l$ 
to
$(\beta_{\mu_1}\cdots\beta_{\mu_{l+1}})\in V_{l+1}.$
The resulting labeled Bratteli diagram with 
$\iota$-map
becomes a $\lambda$-graph system over $\Sigma$
which  presents the Dyck shift $D_2$.
It is called 
the Cantor horizon $\lambda$-graph system for $D_2$, 
denoted by $\LCHD2$ in \cite{KM}. 
Throughout the rest of this section,
we denote the $\lambda$-graph system
$\LCHD2$ by ${\frak L} =(V,E,\lambda,\iota)$ for brevity.

Let $G$ be the finite abelian group
${\Z_2}=\{0,1\} =\Z/2\Z$
of order two.
We will define a map $\ell: \Sigma\rightarrow {\Z_2}$
by
\begin{equation}
\ell(\alpha_1) = \ell(\beta_1) =1, \qquad
\ell(\alpha_2) = \ell(\beta_2) =0. \label{eq:dyckell}
\end{equation}
Let us consider the extension 
${\frak L}^{\Z_2}
=(V^{\Z_2}, E^{\Z_2}, \iota^{\Z_2}, \lambda^{\Z_2}) $
of the $\lambda$-graph system 
${\frak L}(=\LCHD2)$
by the map $\ell$.
The alphabet ${\Sigma}^{\Z_2}$
of ${\frak L}^{\Z_2}$
is
$\Z_2 \times \Sigma$,
and
$V^{\Z_2}= \Z_2 \times V_l,
 E^{\Z_2}= \Z_2 \times E_{l,l+1}.
$
In what follows,
we will study the extension 
${\frak L}^{\Z_2}$.
The map
$
\varphi_{V}:V^{\Z_2}_l\rightarrow V_{l+1}
$
is given by the composition
$
\varphi_{V}
=
\xi_{V}\circ \eta_{V}
$
of the maps
$$
\eta_{V}:V^{\Z_2}_l\rightarrow V_{l+1},
\qquad
\xi_{V}:V_{l+1}\rightarrow V_{l+1}
$$
defined below.
The former $\eta_V$ is defined by
\begin{equation*}
\eta_{V}(g,\gamma_1\cdots\gamma_l) =
\begin{cases}
(\beta_2\gamma_1\cdots\gamma_l) & \text{ if } g =0, \\
(\beta_1\gamma_1\cdots\gamma_l) & \text{ if } g =1 
\end{cases}
\end{equation*}
for
$
g \in \Z_2 =\{0,1\}, 
(\gamma_1\cdots\gamma_l)\in B_l(\Sigma_+).
$
The latter $\xi_V$ is defined by the successive operations 
of the words of length two:
\begin{equation*}
\beta_1 \beta_1 \rightarrow \beta_1 \beta_2,
\qquad
\beta_1 \beta_2 \rightarrow \beta_1 \beta_1
\end{equation*}
from the leftmost of words in 
$V_{l+1}$.
We denote by
$\overline{\beta_1 \beta_1}= \beta_1 \beta_2,
\overline{\beta_1 \beta_2}= \beta_1 \beta_1. 
$
The operation $\xi_V$ 
acts on words of $V$ from the leftmost
successively such as
\begin{align*}
          \beta_1 \beta_2\beta_2\beta_1\beta_1\beta_1 \beta_2 \beta_1
&
\rightarrow
\overline{\beta_1 \beta_2}\beta_2\beta_1\beta_1\beta_1 \beta_2 \beta_1 
=         \beta_1 \beta_1 \beta_2\beta_1\beta_1\beta_1 \beta_2 \beta_1 \\
&
\rightarrow
          \beta_1 \overline{\beta_1\beta_2}\beta_1\beta_1\beta_1 \beta_2 \beta_1 
=
          \beta_1           \beta_1\beta_1 \beta_1\beta_1\beta_1 \beta_2 \beta_1 \\
&
\rightarrow
          \beta_1           \beta_1\overline{\beta_1 \beta_1}\beta_1\beta_1 \beta_2 \beta_1 
=
          \beta_1           \beta_1          \beta_1 \beta_2 \beta_1\beta_1 \beta_2 \beta_1 \\
&
\rightarrow
          \beta_1           \beta_1          \beta_1 \beta_2 
\overline{\beta_1\beta_1} \beta_2 \beta_1 
=
          \beta_1           \beta_1          \beta_1 \beta_2 
          \beta_1\beta_2 \beta_2 \beta_1.   
\end{align*}
We next define a map 
$
\varphi_V:\Sigma^{\Z_2}(=\{0,1\}\times \Sigma)\rightarrow \Sigma
$ 
as follows.
\begin{gather*}
\varphi_\Sigma(0,\beta_2) = \varphi_\Sigma(0,\beta_1) =\beta_2,\qquad 
\varphi_\Sigma(1,\beta_2) = \varphi_\Sigma(1,\beta_1) =\beta_1, \\ 
\varphi_\Sigma(0,\alpha_2) = \varphi_\Sigma(1,\alpha_1) =\alpha_2,\qquad 
\varphi_\Sigma(1,\alpha_2) = \varphi_\Sigma(0,\alpha_1) =\alpha_1
\end{gather*}
\begin{lemma}
There exists a bijective correspondence
$\varphi_E: E_{l-1,l}^{\Z_2} \rightarrow E_{l,l+1}$
satisfying
\begin{equation*}
\varphi_V(s(e^g)) = s(\varphi_E(e^g)), \qquad
\varphi_V(t(e^g)) = t(\varphi_E(e^g)), \qquad
\varphi_\Sigma(\lambda^{\Z_2}(e^g)) = \lambda(\varphi_E(e^g)), 
\end{equation*}
for $e^g = (g,e) \in E_{l-1,l}^{\Z_2} = \Z_2 \times E_{l-1,l}.$
\end{lemma}
\begin{proof}
Tak an arbitrary fixed ege
$e^g = (g,e) \in E_{l-1,l}^{\Z_2} = \Z_2 \times E_{l-1,l}.$
The edge 
$(g,e)$ has four cases for 
$g=0,1, \, \lambda(e) = \alpha_j,\beta_j, j=1,2$
in the following way.

Case 1: $(g,\lambda(e)) = (0, \beta_j), j=1,2$.

\noindent
There exists $\gamma \in B_{l-2}(\Sigma_+)$ and $\eta \in B_2(\Sigma_+)$
such that 
$s(e) = \beta_j \gamma \in V_{l-1}$
and
$t(e) = \gamma\eta \in V_{l}$.
We then have
\begin{align*}
s(e^g) & = (0,s(e)) = (0, \beta_j \gamma),\\
t(e^g) & = (0 + \ell_{\Z_2}(e),t(e)) 
         = (0 + \ell(\beta_j), \gamma\eta)
         = (j, \gamma\eta)
\end{align*}
so that 
\begin{align*}
\varphi_V(s(e^g)) &  =\varphi_V(0, \beta_j \gamma) 
                     =\xi_{V}(\beta_2\beta_j \gamma)
                     =\beta_2\beta_j\xi_{V}( \gamma),\\
\varphi_V(t(e^g)) &  =\varphi_V(j, \gamma \eta) 
                     =\xi_{V}(\beta_j\gamma\eta)
                     =\beta_j\xi_{V}( \gamma\eta).
\end{align*}
Hence there exists a unique edge written 
$\varphi_E(e^g)$ in $E_{l,l+1}$
such that 
\begin{gather*}
s(\varphi_E(e^g))   =\beta_2\beta_j\xi_{V}( \gamma),\qquad
t(\varphi_E(e^g))   =\beta_j\xi_{V}( \gamma\eta),\\
\lambda(\varphi_E(e^g))  =\beta_2  =\varphi_\Sigma(\lambda^{\Z_2}(e^g)).
\end{gather*}

%
 
Case 2: $(g,\lambda(e)) = (1, \beta_j), j=1,2$.

\noindent
There exists $\gamma \in B_{l-2}(\Sigma_+)$ and $\eta \in B_2(\Sigma_+)$
such that 
$s(e) = \beta_j \gamma \in V_{l-1}$
and
$t(e) = \gamma\eta \in V_{l}$.
We then have
\begin{align*}
s(e^g) & = (1,s(e)) = (1, \beta_j \gamma),\\
t(e^g) & = (1 + \ell_{\Z_2}(e),t(e)) 
         = (1 + \ell(\beta_j), \gamma\eta)
         = (1+j, \gamma\eta)
\end{align*}
so that 
\begin{align*}
\varphi_V(s(e^g)) &  =\varphi_V(1, \beta_j \gamma) 
                     =\xi_{V}(\beta_1\beta_j \gamma)
                     =\beta_1\beta_{1+j}\xi_{V}(\gamma),\\
\varphi_V(t(e^g)) &  =\varphi_V(1+j, \gamma \eta) 
                     =\xi_{V}(\beta_{1+j}\gamma\eta)
                     =\beta_{1+j}\xi_{V}( \gamma\eta).
\end{align*}
Hence there exists a unique edge written 
$\varphi_E(e^g)$ in $E_{l,l+1}$
such that 
\begin{gather*}
s(\varphi_E(e^g))   =\beta_1\beta_{1+j}\xi_{V}( \gamma),\qquad
t(\varphi_E(e^g))   =\beta_{1+j}\xi_{V}( \gamma\eta),\\
\lambda(\varphi_E(e^g))  =\beta_1  =\varphi_\Sigma(\lambda^{\Z_2}(e^g)).
\end{gather*}

%
 
Case 3: $(g,\lambda(e)) = (0, \alpha_j), j=1,2$.

\noindent
There exists $\zeta \in B_{l-1}(\Sigma_+)$ 
such that 
$s(e) = \zeta \in V_{l-1}$
and
$t(e) = \beta_j\zeta \in V_{l}$.
We then have
\begin{align*}
s(e^g) & = (0,s(e)) = (0, \zeta),\\
t(e^g) & = (0 + \ell_{\Z_2}(e),t(e)) 
         = (0 + \ell(\alpha_j), \beta_j\zeta)
         = (j, \beta_j\zeta)
\end{align*}
so that 
\begin{align*}
\varphi_V(s(e^g)) &  =\varphi_V(0, \zeta) 
                     =\xi_{V}(\beta_2\zeta)
                     =\beta_2\xi_{V}( \zeta),\\
\varphi_V(t(e^g)) &  =\varphi_V(j, \beta_j \zeta) 
                     =\xi_{V}(\beta_j\beta_j\zeta)
                     =\beta_j\beta_2\xi_{V}(\zeta).
\end{align*}
Hence there exists a unique edge written 
$\varphi_E(e^g)$ in $E_{l,l+1}$
such that 
\begin{gather*}
s(\varphi_E(e^g))   =\beta_2\xi_V( \zeta),\qquad
t(\varphi_E(e^g))   =\beta_j\beta_2\xi_V( \zeta),\\
\lambda(\varphi_E(e^g)) =\alpha_j  =\varphi_\Sigma(\lambda^{\Z_2}(e^g)).
\end{gather*}

Case 4: $(g,\lambda(e)) = (1, \alpha_j), j=1,2$.

\noindent
There exists $\zeta \in B_{l-1}(\Sigma_+)$ 
such that 
$s(e) = \zeta \in V_{l-1}$
and
$t(e) = \beta_j\zeta \in V_{l}$.
We then have
\begin{align*}
s(e^g) & = (1,s(e)) = (1, \zeta),\\
t(e^g) & = (1 + \ell_{\Z_2}(e),t(e)) 
         = (1 + \ell(\alpha_j), \beta_1\zeta)
         = (1+j, \beta_j\zeta)
\end{align*}
so that 
\begin{align*}
\varphi_V(s(e^g)) &  =\varphi_V(1, \zeta) 
                     =\xi_{V}(\beta_1\zeta)
                     =\beta_1\xi_{V}( \zeta),\\
\varphi_V(t(e^g)) &  =\varphi_V(1+j, \beta_j \zeta) 
                     =\xi_{V}(\beta_{1+j}\beta_j\zeta)
                     =\beta_{1+j}\beta_1\xi_{V}(\zeta).
\end{align*}
Hence there exists a unique edge written 
$\varphi_E(e^g)$ in $E_{l,l+1}$
such that 
\begin{gather*}
s(\varphi_E(e^g))   =\beta_1\xi_V( \zeta),\qquad
t(\varphi_E(e^g))   =\beta_{1+j}\beta_1\xi_V( \zeta),\\
\lambda(\varphi_E(e^g)) =\alpha_{1+j} =\varphi_\Sigma(\lambda^{\Z_2}(e^g)).
\end{gather*}

It is routine to check that
the correspondence
$\varphi_E: E_{l-1,l}^{\Z_2} \rightarrow E_{l,l+1}$
satisfies the desired properties.  
\end{proof}

Let us denote by 
${\frak L}^{Ch(D_2)}_1$ the $\lambda$-graph system
$(V^{Ch(D_2)}_1, E^{Ch(D_2)}_1, {\lambda}^{Ch(D_2)}_1,{\iota}^{Ch(D_2)}_1)
$
obtained from
${\frak L}^{Ch(D_2)}$
by shifting ${\frak L}^{Ch(D_2)}$
upward in one-step.
Their vertices and edge sets are defined by
\begin{equation*}
V^{Ch(D_2)}_{1,l} = V^{Ch(D_2)}_{l+1}, \qquad
E^{Ch(D_2)}_{1,l,l+1} = E^{Ch(D_2)}_{l+1,l+2}, \qquad
l \in \Zp.
\end{equation*}
The maps ${\lambda}^{Ch(D_2)}_1,{\iota}^{Ch(D_2)}_1$
are induced from
${\lambda}^{Ch(D_2)},{\iota}^{Ch(D_2)}$
in a natural way, respectively.
We know the following proposition from the above lemma.
\begin{proposition}
The $\Z_2$-extension $({\frak L}^{Ch(D_2)})^{\Z_2,\ell}$ 
of the Cantor horizon $\lambda$-graph system
${\frak L}^{Ch(D_2)}$ by the function $\ell$
defined by \eqref{eq:dyckell}
is isomorphic to the $\lambda$-graph system 
${\frak L}^{Ch(D_2)}_1$ obtained from 
${\frak L}^{Ch(D_2)}$ by shifting upward in one-step
up to labeling.
\end{proposition}
Let us denote by $\tau:D_2\rightarrow \Z_2$
the skewing function defined by
$\tau((x_n)_{n\in \Z}) =\ell(x_0)$ for 
$(x_n)_{n\in \Z}$.
By Theorem \ref{thm:G-conj},
the $\Z_2$-extension $D_2^{\Z_2,\tau}$
of $D_2$ by $\tau$ is presented by the 
the $\Z_2$-extension 
${\frak L}^{Ch(D_2),\ell}$ 
of the $\lambda$-graph system
${\frak L}^{Ch(D_2)}$ 
by $\ell$.
The map $\varphi_E:E_{l,l+1}^{\Z_2}$
induces a factor map from the 
$\Z_2$-extension 
$D_2^{\Z
_2,\tau}$
onto $D_2$.

\medskip

{\bf 2.}
Let $A=[A(i,j)]_{i,j=1}^N$ be an  $N\times N$ 
matrix with entries in nonnegative integers.
The matrix defines a finite directed graph $\G_A =(V_A, E_A)$
with vertex set $V_A$ having $N$ vertices
and edge set $E_A = \{ e_1,\dots,e_{N_A}\}$
 such that
the number of the edges from $v_i$ to $v_j$ is $A(i,j)$ for $i,j=1,\dots,N$.
The Markov--Dyck shifts are generalization of Dyck shifts, 
and their concatenating bracket rules 
come from the directed graph 
$\G_A$ (see \cite{HIK}, \cite{KrBLMS}, \cite{KMMunster}, \cite{MaMathScand2011}).
We will briefly review the Markov--Dyck shift $D_A$ for the matrix $A$.
Let us denote by $A^\G$
the $N_A \times N_A$ transition matrix defined by
\begin{equation}
A^\G(i,j) = 
\begin{cases}
1 & \text{ if } t(e_i) = s(e_j), \\
0 & \text{ otherwise}.
\end{cases} \label{eq:AG}
\end{equation} 
Let $S_i, i=1,\dots, N_A$ be a family of partial isometries 
satisfying the relations 
\begin{equation}
\sum_{j=1}^{N_A} S_j S_j^* =1,\qquad 
S_i^* S_i = \sum_{j=1}^{N_A} A^{\G}(i,j) S_j S_j^*,
\qquad i=1,\dots,N_A. \label{eq:CKAG}
\end{equation}
They are generating family of the Cuntz--Krieger algebra 
${\mathcal{O}}_{A^\G}$.
We set the symbols 
$$
\alpha_i = S_i^*, \qquad
\beta_i = S_i, \qquad i=1,\dots, N_A
$$
and
$\Sigma^- = \{\alpha_1,\dots,\alpha_{N_A}\}, \Sigma^+ = \{\beta_1,\dots,\beta_{N_A}\},
\Sigma = \Sigma^- \cup \Sigma^+.
$ 
Let ${\frak F}_A$ be the set of words $\gamma_1,\dots,\gamma_n$
of $\Sigma$ such that 
the  product $\gamma_1\cdots\gamma_n$ in the algebra ${\mathcal{O}}_{A^\G}$
is $0$.
The Markov--Dyck shift $D_A$ is defined by a subshift over $\Sigma$
whose forbidden words are ${\frak F}_A$.
If $A =[N]$ the $1 \times 1$ matrix with entry $N>1$.
The Markov--Dyck shift $D_A$ becomes the Dyck shift $D_N$.
If $A$ is irreducible and not any permutation matrix, 
the subshift $D_A$ is not sofic (\cite{MaMathScand2011}).  
Since 
$S_{i_1}\cdots S_{i_n} \ne 0$ if and only if 
$A^{\G}(i_1,i_2)\cdots A^{\G}(i_{n-1},i_n) \ne 0$,
the subset of $D_A$ consisting of biinfinite sequences of $\Sigma^+$
forms the SFT $\Lambda_A$ for the original matrix $A$.   
Similarly the subset $D_A$ consisting of biinfinite sequences of $\Sigma^-$
forms the SFT $\Lambda_{A^t}$ for the transposed matrix $A^t$ of $A$.
 Hence the Markov--Dyck shift $D_A$
contains the two SFTs $\Lambda_A$ and $\Lambda_{A^t}$
 as its subsystems.

Let $G$ be a finite group.
Suppose that maps
$\ell_{-}: \Sigma^- \longrightarrow G,\,
 \ell_{+}: \Sigma^+ \longrightarrow G$ 
are given, so that the map combined with them
$\ell: \Sigma \longrightarrow G$
is obtained.
Define
$\tau_{\ell}: D_A \longrightarrow G$
by setting
$\tau_{\ell}((\gamma_n)_{n \in \Z}) = \ell(\gamma_0)$.
As in Section 3, the skewing function 
$\tau_{\ell}$ defines a $G$-subshift
$D_A^{G,\tau_{\ell}} (= (G\ltimes D_A, \tau_{\ell}\ltimes \sigma_{D_A})).$
Since both of the SFTs $\Lambda_A$ and $\Lambda_{A^t}$
are subsystems of $D_A$, the restrictions
of $\tau_\ell$ to them yield their skewing functions written
 $\tau_{\ell_A}$ and $\tau_{\ell_{A^t}}$, respectively,
 so that we have two $G$-SFTs   
$\Lambda_A^{G,\tau_{\ell_A}}$ and $\Lambda_{A^t}^{G,\tau_{\ell_{A^t}}}$
from the skewing functions.

Let $A'$ be another irreducible square 
matrix with entries in nonnegative integers.
We similarly have a directed graph $\G_{A'} =(V_{A'}, E_{A'})$
and symbol sets 
$
{\Sigma'}^- = \{\alpha'_1,\dots,\alpha'_{N_{A'}}\}, 
{\Sigma'}^+ = \{\beta'_1,\dots,\beta'_{N_{A'}}\},
\Sigma' = {\Sigma'}^- \cup {\Sigma'}^+.
$ 
We have the  Markov--Dyck shift $D_{A'}$ 
and two SFTs $\Lambda_{A'}$ and $\Lambda_{{A'}^t}$ as its subsystems.
W. Krieger has shown that the following classification result for Markov--Dyck shifts.
\begin{proposition}[{\cite[Corollary 3.2]{KrBLMS}}, cf. \cite{HIK}]
Let $A$ and $A'$ be two irreducible square matrices with entries in nonnegative integers.
Assume that each vertex of both of their directed graphs 
$\G_A$ and $\G_{A'}$ has  at least two in-coming edges.
Then the Markov--Dyck shifts 
$D_A$ and $D_{A'}$ are topologically conjugate 
if and only if the directed graphs  
$\G_A$ and $\G_{A'}$  are isomorphic as directed graphs.
\end{proposition}

For matrix $A'$, suppose that maps
$\ell'_{-}: {\Sigma'}^- \longrightarrow G,\,
 \ell'_{+}: {\Sigma'}^+ \longrightarrow G$ 
are given, so that the map combined with them
$\ell': \Sigma \longrightarrow G$
is obtained, and we have 
a $G$-subshift
$D_{A'}^{G,\tau_{\ell'}} $
and two $G$-SFTs   
$\Lambda_{A'}^{G,\tau_{\ell_{A'}}}$ and 
$\Lambda_{{A'}^t}^{G,\tau_{\ell_{{A'}^t}}}$
as seen in the preceding discussions.
By using the above proposition,
we see the following lemma.
\begin{lemma}
Assume that each vertex of both of the directed  graphs 
$\G_A$ and $\G_{A'}$ has  at least two in-coming edges.
If there exists a topological conjugacy 
$\Phi: D_A \longrightarrow D_{A'}$ such that 
$\tau_\ell$ is cohomologous to $\tau_{\ell'} \circ \Phi$ in $C(D_A,G)$,
then there exists a topological conjugacies 
$\Phi_+: \Lambda_A \longrightarrow \Lambda_{A'}$ 
and
$\Phi_-: \Lambda_{A^t} \longrightarrow \Lambda_{A'^t}$ 
such that 
$\tau_{\ell_A}$ on $\Lambda_A$ is cohomologous to $\tau_{\ell'_{A'}} \circ \Phi_+$ in $C(\Lambda_A,G)$
and
$\tau_{\ell_{A^t}}$on $\Lambda_{A^t}$  is cohomologous to $\tau_{\ell'_{{A'}^t}} \circ \Phi_-$ in $C(\Lambda_{A^t},G).$
\end{lemma}
\begin{proof}
By the preceding proposition and its proof in \cite[Corollary 3.2]{KrBLMS},
the conjugacy
$\Phi: D_A \longrightarrow D_{A'}$ 
is induced by a one-block map which gives rise to topological conjugacies
$\Lambda_A \longrightarrow \Lambda_{A'}$ 
and
$\Lambda_{A^t} \longrightarrow \Lambda_{A'^t}$ 
of its subsystems as SFTs.
They are given by restricting $\Phi$ to 
$\Lambda_A $ and $\Lambda_{A'}$, respectively,
so that
$\tau_{\ell_A}$ on $\Lambda_A$ is cohomologous to $\tau_{\ell'_{A'}} \circ \Phi_+$ in $C(\Lambda_A,G)$
and
$\tau_{\ell_{A^t}}$ on $\Lambda_{A^t}$
 is cohomologous to $\tau_{\ell'_{{A'}^t}} \circ \Phi_-$ in $C(\Lambda_{A^t},G).$
\end{proof}
We thus have  
\begin{proposition}
Assume that each vertex of both of the directed graphs 
$\G_A$ and $\G_{A'}$ has  at least two in-coming edges.
If $G$-subshifts 
$D_A^{G,\tau_{\ell}}$ and 
$D_{A'}^{G,\tau_{\ell'}} $
are $G$-conjugate,
then 
the two $G$-SFTs   
$\Lambda_A^{G,\tau_{\ell_A}}$ and $\Lambda_{A'}^{G,\tau_{\ell'_{A'}}}$
are $G$-conjugate, 
and also 
 the two $G$-SFTs   
$\Lambda_{A^t}^{G,\tau_{\ell_{A^t}}}$ and $\Lambda_{{A'}^t}^{G,\tau_{\ell'_{{A'}^t}}}$
are $G$-conjugate.
\end{proposition}
\begin{proof}
The assertion follows from Theorem \ref{thm:G-conj}.
\end{proof}
\begin{corollary}\label{cor:MDG}
Let $A$ be an $N\times N$ irreducible matrix with entries in nonnegative integers
such that $\sum_{i=1}^N A(i,j) \ge 2$ for each $j=1,\dots,N$.
Then the $G$-conjugacy class of $G$-SFT
  $\Lambda_A^{G,\tau_{\ell_A}}$
is an invariant of the $G$-conjugacy class of $G$-subshift 
$D_A^{G,\tau_{\ell}}$ of the Markov--Dyck shift $D_A$. 
Hence for maps $\ell^1, \ell^2: \Sigma^- \cup \Sigma^+ \longrightarrow G$, 
if two $G$-SFTs   
$\Lambda_A^{G,\tau_{\ell^1_A}}$ and $\Lambda_A^{G,\tau_{\ell^2_A}}$
are not $G$-conjugate,
then
the $G$-subshifts 
$D_A^{G,\tau_{\ell^1}}$ and 
$D_A^{G,\tau_{\ell^2}} $
are not $G$-conjugate.
\end{corollary}
In the remainder of this section, 
we will give an example of a pair of $G$-extensions of  a non sofic subshift  which are not $G$-conjugate
by using Corollary \ref{cor:MDG}.
Let $A =[2]$ the $1\times 1$ matrix with its entry $2$.
The associated directed graph $\G_{[2]} =(V_{[2]}, E_{[2]})$ has one vertex $v$ 
and two directed self-loops $e_1, e_2$ around $v$, so that $V_{[2]} =\{v\}, E_{[2]} =\{ e_1, e_2\}$.
The subshift $\Lambda_A$ of the directed graph is the full $2$-shift $\Lambda_{[2]}$.
Let $G = \Z_2 =\{0,1\}$.
 Define $\ell^i:E_{[2]} \longrightarrow \Z_2$ for $i=1,2$ by setting
\begin{equation*}
\ell^1(e_1) = \ell^1(e_2) = 0, \qquad
\ell^2(e_1) = 1, \quad
\ell^2(e_2) = 0.
\end{equation*} 
We will first see that the maps
$\tau_{\ell^1}, \tau_{\ell^2} : \Lambda_{[2]} \longrightarrow \Z_2$
defined by $\tau_{\ell^i}((x_n)_{n \in \Z}) = \ell^i(x_0), i=1,2$
are not cohomologous  to each other.
Suppose that there exists a continuous function 
$\eta: \Lambda_{[2]} \longrightarrow \Z_2$
such that 
\begin{equation*}
\tau_{\ell^2}(x) = \eta(x) + \tau_{\ell^1}(x) -\eta(\sigma_{[2]}(x)), \qquad 
x =(x_n)_{n \in \Z} \in \Lambda_{[2]}.
\end{equation*} 
Take the element $x=(x_n)_{n \in \Z} \in \Lambda_{[2]}$ 
such that $ x_n = e_1, n\in \Z$, so that
we have
$\tau_{\ell^2}(x) =\ell^2(x_0) = \ell^2(e_1) =1$
and
$\tau_{\ell^1}(x) = \ell^1(x_0) =\ell^1(e_1) =0.$
Hence we have 
$\eta(x) + \tau_{\ell^1}(x) -\eta(\sigma_{[2]}(x)) = \eta(x) -\eta(x) =0$,
a contradiction.
Hence $\tau_{\ell^1}$ is not cohomologous to $\tau_{\ell^2}$.
By Theorem \ref{thm:BS},
the G-SFTs 
$\Lambda_{[2]}^{\Z_2,\tau_{\ell^1}}$ and
$\Lambda_{[2]}^{\Z_2,\tau_{\ell^2}}$
are not $G$-conjugate.
Let us consider the Markov--Dyck shift 
$D_A$ for $A =[2]$. 
In this case,
it is the Dyck shift $D_2$.
Define
$\ell^i_{-}: \Sigma^- \longrightarrow \Z_2
$
and
$\ell^i_{+}: \Sigma^- \longrightarrow \Z_2
$
for $i=1, 2$ by setting
\begin{gather}
\ell^1_{-}(\alpha_1) = 
\ell^1_{-}(\alpha_2) = 
\ell^1_{+}(\beta_1) = 
\ell^1_{+}(\beta_2) = 0, \label{eq:l1}\\
\ell^2_{-}(\alpha_1) = 1, \qquad
\ell^2_{-}(\alpha_2) = 0, \qquad
\ell^2_{+}(\beta_1) = 1, \qquad
\ell^2_{+}(\beta_2) = 0. \label{eq:l2}
\end{gather} 
The restrictions of the skewing functions 
$\tau_{\ell^i}$ on $D_A$ to the full $2$-shift $\Lambda_{[2]}$
are the same as the above skewing functions 
$\tau_{\ell^i}$ on $\Lambda_{[2]}$.
Since the G-SFTs 
$\Lambda_{[2]}^{\Z_2,\tau_{\ell^1}}$ and
$\Lambda_{[2]}^{\Z_2,\tau_{\ell^2}}$
are not $G$-conjugate,
Corollary \ref{cor:MDG} ensures us
the following Proposition.
\begin{proposition} \label{prop:7.7}
Let $\ell^1, \ell^2: \{\alpha_1,\alpha_2\}\cup \{\beta_1,\beta_2\} \rightarrow \Z_2$
be the map defined by \eqref{eq:l1}, \eqref{eq:l2},
and $\tau_{\ell^1}, \tau_{\ell^2}$ the associated skewing functions on the Dyck shift $D_2$.
Then the $G$-subshifts
$D_2^{\Z_2,\tau_{\ell^1}}$
and
$D_2^{\Z_2,\tau_{\ell^2}}$
are not $G$-conjugate for $G=\Z_2$.
\end{proposition}

By the same idea as above, 
we have many non $G$-conjugate $G$-extensions of a Markov--Dyck shift
from a family of non $G$-conjugate $G$-extensions of SFTs.
We note that non $G$-conjugate SFTs are studied in \cite{BoyleSchmieding} 
(cf. \cite{BoyleSullivan}).

\section{Concluding Remark}
As in \cite{DocMath2002}, 
a general $\lambda$-graph system ${\frak L}$
naturally gives rise to a $C^*$-algebra ${\mathcal{O}}_{\frak L}$.
The class of these $C^*$-algebras ${\mathcal{O}}_{\frak L}$
is a generalized class of Cuntz--Krieger algebras ${\mathcal{O}}_A$,
which are associated to finite directed graphs with transition matrix $A$.
Let ${\frak L} =(V,E,\lambda,\iota)$
be a $\lambda$-graph system over $\Sigma$
and $G$ be a finite group.
Suppose that a map $\ell: \Sigma \rightarrow G$
is given.
In this final section, we remark that
there is a relationship
between the two $C^*$-algebras
$\OL$ and $\OLG$.

Let us now briefly review the $C^*$-algebra
$\OL$ associated with $\lambda$-graph system ${\frak L}$.
We denote by 
$\{v_1^l,\dots,v_{m(l)}^l\}$ 
the vertex set $V_l$.
Define the structure matrices $A_{l,l+1}, I_{l,l+1}$
of ${\frak L}$
by setting
for
$
i=1,2,\dots,m(l),\ j=1,2,\dots,m(l+1), \ \alpha \in \Sigma,
$ 
\begin{align*}
A_{l,l+1}(i,\alpha,j)
 & =
{
\begin{cases}
1 &  
    \text{ if } \ s(e) = v_i^l, \lambda(e) = \alpha,
                       t(e) = v_j^{l+1} 
    \text{ for some }    e \in E_{l,l+1}, \\
0           & \text{ otherwise,}
\end{cases}
} \\
I_{l,l+1}(i,j)
 & =
{
\begin{cases}
1 &  
    \text{ if } \ \iota_{l,l+1}(v_j^{l+1}) = v_i^l, \\
0           & \text{ otherwise.}
\end{cases}
} 
\end{align*}
The $C^*$-algebra $\OL$
is realized as the universal unital $C^*$-algebra
generated by
partial isometries
$S_{\alpha}, \alpha \in \Sigma$
and projections
$E_i^l, i=1,2,\dots,m(l),\l\in \Zp 
$
 subject to the  following operator relations called $({\frak L})$:
\begin{align}
\sum_{\beta \in \Sigma} S_{\beta} S_{\beta}^*  & = 1,
\label{eq:SS}  \\
 \sum_{i=1}^{m(l)} E_i^l   =  1, \qquad 
 E_i^l & = \sum_{j=1}^{m(l+1)} I_{l,l+1}(i,j)E_j^{l+1},  
 \label{eq:EIE} \\
 S_\alpha S_\alpha^* E_i^l & =   E_i^{l} S_\alpha S_\alpha^*, 
 \label{eq:SSE} \\
S_\alpha^*E_i^l S_\alpha  =  
\sum_{j=1}^{m(l+1)} & A_{l,l+1}(i,\alpha,j)E_j^{l+1}, \label{eq:SES}
\end{align}
for $\alpha \in \Sigma,$
$
i=1,2,\dots,m(l),\l\in \Zp. 
$
It is nuclear (\cite[Proposition 5.6]{DocMath2002}).
Under the condition (I) defined in \cite{DocMath2002}, 
the algebra $\OL$ can be realized as 
the unique 
$C^*$-algebra subject to  the relations $({\frak L})$
(\cite[Theorem 4.3]{DocMath2002}).
If ${\frak L}$ has some irreduciblity with  condition (I),
the $C^*$-algebra $\OL$ is simple
(\cite[Theorem 4.7]{DocMath2002}, cf. \cite{MathScand2005}).


Suppose that a map
$\ell:\Sigma\rightarrow G$
is given.
Let us denote by $M_{|G|}({\mathbb{C}})$
the $|G|\times |G|$ full matrix algebra.
Let 
$\{e_{g,h}\}_{g,h \in G}$
be the system of matrix units of 
$M_{|G|}({\mathbb{C}})$.
We put $e_g , g \in G$ the diagonal matrix 
$e_{g,g}$ having $1$ only at $(g,g)$-component
and $0$ elsewhere.
We fix the canonical generators
$S_\alpha, E_i^l$ of $\OL$
satisfying the relations $({\frak L})$. 
We set
\begin{align*}
S_{(g,\alpha)} & = e_{g,g\ell(\alpha)}\otimes S_\alpha 
\quad \text{ for } (g,\alpha) \in \Sigma^G, \\
E_{(g,v_i^l)} & = e_{g}\otimes E_i^l 
\quad \text{ for } (g,v_i^l) \in V_l^G.
\end{align*}
It is easy to see that the following identities hold:
\begin{align}
\sum_{(g,\beta)\in \Sigma^G} S_{(g,\beta)}S_{(g,\beta)}^* & = 1, 
\label{eq:GSS} \\
\sum_{(g,v_i^l) \in V_l^G} E_{(g,v_i^l)}   =  1, \qquad 
E_{(g,v_i^l)} & = \sum_{j=1}^{m(l+1)} I_{l,l+1}(i,j)E_{(g,v_j^{l+1})}, 
\label{eq:GEIE} \\
 S_{(g,\alpha)} S_{(g,\alpha)}^* E_{(h,v_i^l)} 
 & =  E_{(h,v_i^l)} S_{(g,\alpha)} S_{(g,\alpha)}^*, 
 \label{eq:GSSE} \\
S_{(g,\alpha)}^*  E_{(g,v_i^l)}  S_{(g,\alpha)}  =  
\sum_{j=1}^{m(l+1)} & A_{l,l+1}(i,\alpha,j)E_{(g\ell(\alpha),v_j^{l+1})}
\label{eq:GSES}
\end{align}
for $(g,\alpha) \in \Sigma^G,(g,v_i^l)\in V_l^G, l\in \Zp$. 

Let
$A^G, I^G$ be the structure matrices of the $\lambda$-graph system
${\frak L}^{G,\ell}$.
They are related to those $A,I$
of ${\frak L}$ in the following way.
For 
$
(g,v_i^l) \in V_l^G, 
(g',v_j^{l+1}) \in V_{l+1}^G  
$
and
$(h,\alpha)\in \Sigma^G$,
we have
\begin{align*}
A_{l,l+1}^G((g,v_i^l), (h,\alpha),(g',v_j^{l+1}) )
&
= 
{
\begin{cases}
A_{l,l+1}(i,\alpha,j) & \text{ if } g=  h \text{ and } g' = g \ell(\alpha), \\
0 & \text{ otherwise, }
\end{cases}
} \\
I_{l,l+1}^G((g,v_i^l), (g',v_j^{l+1}) )
&
= 
{
\begin{cases}
I_{l,l+1}(i,j) & \text{ if } g= g', \\
0 & \text{ otherwise.}
\end{cases}
}
\end{align*}
Hence we know that the relations 
\eqref{eq:GSS},
\eqref{eq:GEIE},
\eqref{eq:GSSE}
and
\eqref{eq:GSES}
become 
the relations
\eqref{eq:SS},
\eqref{eq:EIE},
\eqref{eq:SSE}
and
\eqref{eq:SES}
 for the $\lambda$-graph system
 ${\frak L}^{G,\ell}$.
 It is easy to see that 
the extension 
${\frak L}^{G,\ell}$ 
satisfies condition (I),
if  
${\frak L}$ satisfies condition (I).
We thus conclude that
the $C^*$-algebra
$C^*(S_{(g,\alpha)}, E_{(g,v_i^l)}; \alpha \in \Sigma,\ v_i^l \in V,\ g \in G)
$  
generated by the partial isometries
$S_{(g,\alpha)}$
and the projections 
$
E_{(g,v_i^l)} 
$
is canonically isomorphic to
the $C^*$-algebra
$\OLG$.

Define an action $\gamma$ of $G$ on $\OLG$ by setting
\begin{equation*}
\gamma_h(S_{(g,\alpha)}) =S_{(hg,\alpha)}, \qquad
\gamma_h(E_{(g,v_i^l)}) = E_{(hg,v_i^l)} \quad
\text{ for } \alpha \in \Sigma,\ v_i^l \in V,\ g, h \in G.
\end{equation*}
The action $\gamma$ on $\OLG$ comes from the action of $G$ on the $G$-$\lambda$-graph system ${\frak L}^{G,\ell}$ defined in Section 4.
The following lemma will be used to prove Proposition \ref{prop:8.2}.
\begin{lemma}\label{lem:outer}
The automorphism $\gamma_h$ on $\OLG$ is outer for each $h \in G$ with $h \ne 1$.
\end{lemma}
\begin{proof}
Fix $h \in G$ with $h \ne 1$.
Suppose that $\gamma_h$ is inner so that there exists a unitary 
$U$ in $\OLG$ such that $\gamma_h = \Ad(U)$.
For an admissible word 
$\mu =(\mu_1,\dots,\mu_n) \in B_*(\Lambda_{{\frak L}^{G,\ell}})$
where $\mu_i = (g_i,\alpha_i)$ for some $g_i \in G, \alpha_i \in \Sigma$
of the presented subshift 
$\Lambda_{{\frak L}^{G,\ell}}$,
denote by  
$S_\mu$ 
the partial isometry 
$S_{(g_1,\alpha_1)}\cdots S_{(g_n,\alpha_n)}$ in $\OLG$.
We define the $C^*$-subalgebras $\FLG, \DLG$ of $\OLG$
by setting
\begin{align*}
\FLG &= C^*(S_\mu E_{(g,v_i^l)}S_\nu^* ; (g,v_i^l) \in V^G,
                   \mu, \nu \in B_*(\Lambda_{{\frak L}^{G,\ell}}) \text{ with } |\mu| = |\nu|), \\
\DLG &= C^*(S_\mu E_{(g,v_i^l)}S_\mu^* ; (g,v_i^l) \in V^G,
                   \mu \in B_*(\Lambda_{{\frak L}^{G,\ell}}) ). 
\end{align*}
They are the $C^*$-subalgebras of $\OLG$ generated by 
elements of the form 
$
S_\mu E_{(g,v_i^l)}S_\nu^*
$ 
for 
$(g,v_i^l) \in V^G, \mu, \nu \in B_*(\Lambda_{{\frak L}^{G,\ell}}) 
$
with
$ |\mu| = |\nu|,$
and
$S_\mu E_{(g,v_i^l)}S_\mu^*
$
for
$(g,v_i^l) \in V^G,\mu \in B_*(\Lambda_{{\frak L}^{G,\ell}}),$
respectively.
The latter subalgebra $\DLG$ is a maximal abelian $C^*$-subalgebra of the AF-algebra
$\FLG$.
Since $\gamma_h(\DLG) = \DLG$,
the unitary $U$ gives rise to an element of the normalizer 
$N(\OLG,\DLG)$.
By \cite{MaYMJ2010} with the condition $\gamma_h(\FLG) = \FLG$, 
there exist $N_0, l_0 \in \N$ such that 
\begin{equation}
U = \sum_{\xi,\eta \in B_{N_0}(\Lambda_{{\frak L}^{G,\ell}}), \, (g,v_i^{l_0})\in V_{l_0}^G }
c_{\xi,(g,v_i^{l_0}),\eta}S_\xi E_{(g,v_i^l)} S_{\eta}^*  \label{eq:U}
\end{equation}
where 
$c_{\xi,(g,v_i^{l_0}),\eta} \in \C$.
Since
$US_{(g,\alpha)}U^* = S_{(hg,\alpha)}$,
we have
\begin{equation}
S_{(hg,\alpha)}^* US_{(g,\alpha)} =S_{(hg,\alpha)}^*S_{(hg,\alpha)}U
\quad \text{ for } (g,\alpha) \in \Sigma^G. \label{eq:SU}
\end{equation}
Take $l_1 >N_0$.
For a fixed $(g,v_i^{l_1}) \in V^G,$
there exist 
$(g_1,\alpha_1), \dots, (g_{N_0},\alpha_{N_0}) \in \Sigma^G$
such that 
\begin{equation}
E_{(g,v_i^{l_1})} 
\le 
S_{(h g_{N_0},\alpha_{N_0})}^*\cdots S_{(h g_1,\alpha_1)}^*
S_{(h g_1,\alpha_1)}\cdots S_{(h g_{N_0},\alpha_{N_0})}. \label{eq:8.11}
\end{equation}
We may in fact find a sequence 
$f_1, f_2, \dots, f_{N_0}$ 
of labeled edges in $\frak L$ such that 
$t(f_i) = s(f_{i+1})$ for $i=1,2,\dots, N_0-1$
and $t(f_{N_0}) = v_i^l$.
Put
\begin{gather*}
g_{N_0} = h^{-1} g \ell_G(f_{N_0})^{-1}, \qquad 
g_{N_0-1} = h^{-1} g \ell_G(f_{N_0})^{-1} \ell_G(f_{N_0}-1)^{-1}, \\ \quad \dots, \quad
g_{1} = h^{-1} g \ell_G(f_{N_0})^{-1} \ell_G(f_{N_0}-1)^{-1}\cdots \ell_G(f_{1})^{-1}, 
\end{gather*}
and
$\alpha_i = \ell_G(f_i)$ for $i=1,\dots, N_0,$
which satisfy the inequality \eqref{eq:8.11}.
By \eqref{eq:SU}, we have
\begin{align*}
& S_{(h g_{N_0},\alpha_{N_0})}^*\cdots S_{(h g_1,\alpha_1)}^* U
  S_{(g_1,\alpha_1)}\cdots S_{(g_{N_0},\alpha_{N_0})} \\
=&
S_{(h g_{N_0},\alpha_{N_0})}^*\cdots S_{(h g_1,\alpha_1)}^* 
S_{(h g_1,\alpha_1)}\cdots S_{(h g_{N_0},\alpha_{N_0})} U.
\end{align*}
By \eqref{eq:U},
the element 
$S_{(h g_{N_0},\alpha_{N_0})}^*\cdots S_{(h g_1,\alpha_1)}^* U
  S_{(g_1,\alpha_1)}\cdots S_{(g_{N_0},\alpha_{N_0})}$    
belongs to
the $C^*$-subalgebra 
${{\mathcal{A}}_{{\frak L}^{G,\ell}}}
$
generated by the projections
$E_{(g, v_i^l)}, (g, v_i^l)\in V^G.$ 
Since
\begin{align*}
 E_{(g,v_i^{l_1})} U 
& =  E_{(g,v_i^{l_1})} S_{(h g_{N_0},\alpha_{N_0})}^*\cdots S_{(h g_1,\alpha_1)}^*
                         S_{(h g_1,\alpha_1)}\cdots S_{(h g_{N_0},\alpha_{N_0})} U \\
& =  E_{(g,v_i^{l_1})}S_{(h g_{N_0},\alpha_{N_0})}^*\cdots S_{(h g_1,\alpha_1)}^* U
                        S_{(g_1,\alpha_1)}\cdots S_{(g_{N_0},\alpha_{N_0})}
\end{align*}
which belongs to ${{\mathcal{A}}_{{\frak L}^{G,\ell}}}$
and
$U  = \sum_{(g,v_i^{l_1}) \in V_{l_1}^G} E_{(g,v_i^{l_1})} U,$ 
the unitary $U$ belongs to the commutative $C^*$-subalgebra 
${{\mathcal{A}}_{{\frak L}^{G,\ell}}},
$
so that 
$ U E_{(g,\alpha)}U^* = E_{(g,\alpha)}$.
This is a contradiction, 
because
$
E_{(hg,\alpha)} = \gamma_h(E_{(g,\alpha)}) = U E_{(g,\alpha)}U^*.
$
\end{proof}
By using the above lemma, we can show the following proposition.
\begin{proposition}\label{prop:8.2}
Suppose that a $\lambda$-graph system ${\frak L}$ is irreducible and satisfies condition (I)
 in the sense of \cite{DocMath2002}.
Then the crossed product $C^*$-algebra
$\OLG\rtimes_{\gamma}G$ is isomorphic to 
the tensor product 
$M_{|G|}({\mathbb{C}}) \otimes \OL$.
\end{proposition}
\begin{proof}
Let $\{ \delta_h \mid h \in G\}$ be the standard basis of $\C^{|G|}$.
Define the unitary $U_h$ on $\C^{|G|}$ for each $h \in G$ by 
$U_h\delta_k = \delta_{hk}, k \in G$.
We set 
$u_h = U_h \otimes 1$ in 
$M_{|G|}({\mathbb{C}}) \otimes \OL$.
Then it is easy to see that 
\begin{equation*}
\gamma_h(S_{(g,\alpha)}) = \Ad(u_h)(S_{(g,\alpha)}), \qquad 
\gamma_h(E_{(g,v_i^l)}) = \Ad(u_h)(E_{(g,v_i^l)}) 
\end{equation*}
for
$h,g \in G, \, \alpha \in \Sigma, \, v_i^l \in V. 
$
By the universality of the crossed product
$\OLG\rtimes_{\gamma}G$,
there exists a natural $*$-homomorphism $\varphi$
from 
$\OLG\rtimes_{\gamma}G$  onto the $C^*$-algebra
$C^*(\OL, u_h; h \in G)$
generated by $\OL$ and the unitaries $u_h, h \in G$.
Since ${\frak L}$ is irreducible and satisfies condition (I),
the $C^*$-algebra $\OL$ is simple by \cite{DocMath2002}. 
As 
automorphisms $\gamma_h$ are outer for $h\ne 1$ by Lemma \ref{lem:outer},
the crossed product
$\OLG\rtimes_{\gamma}G$ is simple by \cite{Kishi}.
Therefore the homomorphism  
$\varphi: \OLG\rtimes_{\gamma}G\longrightarrow C^*(\OL, u_h; h \in G)$
is actually isomorphic.
It is then easy to verify that
the latter $C^*$-algebra $C^*(\OL, u_h; h \in G)$
 is isomorphic to 
the tensor product $C^*$-algebra
$M_{|G|}({\mathbb{C}}) \otimes \OL$.
\end{proof}

Let $A$ be an $N\times N$ irreducible non permutation matrix with entries in nonegative integers.
Let $\G_A=(V_A, E_A)$ be the associated directed graph
with 
vertex set $V_A =\{v_1,\dots,v_N\}$ 
and
edge set $E_A =\{e_1,\dots,e_{N_A}\}$.
Suppose that a map $\ell: E_A \longrightarrow G$ 
is given.
We define a directed graph
$\G_{A^\ell}$
 giving rise to the extension of the SFT
$\Lambda_A$ by the map $\ell$ by following \cite[Sectoin 2]{BoyleSullivan}.
We set 
$V_A^{G,\ell} =\{(g, v_i) \mid g \in G, i=1,\dots,N \}$.
For each $e \in E_A$ from $v_i$ to $v_j$ in $\G_A$
and $g \in G$,
draw an edge from $(g,v_i)$ to $(g\ell(e), v_j)$.
We thus have a directed graph
with vertex set 
$V_A^{G,\ell}$.
It is written $\G_{A^\ell}$,
which gives rise to a $G$-SFT
such that 
the associated $\lambda$-graph system to the  directed graph $\G_{A^\ell}$ is   
the $G$-extension of the $\lambda$-graph system of the original directed graph  
$\G_A$.
The nonnegative matrix for the directed graph $\G_{A^\ell}$ 
is denoted by $A^{G,\ell}$.
Now let $A^\G$ be the $N_A \times N_A$ matrix defined by 
\eqref{eq:AG},
and
$S_1,\dots, S_{N_A}$  a family of partial isometries satisfying the relations
\eqref{eq:CKAG}.
The partial isometries generate the Cuntz--Krieger algebra 
${\mathcal{O}}_{A^\G}$ which is written $\OA$.
Let us denote by
$\{e_{g,h}\}_{g,h\in G}$ the system of matrix units of $M_{|G|}(\C)$.
The Cuntz--Krieger algebra
 ${\mathcal{O}}_{A^{\G,\ell}}$
is generated by the family of the following partial isometries 
\begin{equation*}
S_{(g,i)} = e_{(g,g\ell(e_i))} \otimes S_i, \qquad
g \in G, \, e_i \in E_A, \, i=1,\dots,N_A.
\end{equation*}
We define an action $\gamma$ of $G$ on 
${\mathcal{O}}_{A^{\G,\ell}}$
by
$\gamma_h(S_{(g,i)} ) = S_{(hg,i)},\,  h,g \in G, i=1,\dots,N_A$.
As a corollary of Proposition \ref{prop:8.2}, we have 
\begin{corollary}
The crossed product $C^*$-algebra 
${\mathcal{O}}_{A^{\G,\ell}}\rtimes_{\gamma}G$ is isomorphic to 
the tensor product
$M_{|G|}(\C)\otimes \OA$.
\end{corollary}

\bigskip

{\it Acknowledgment.}
The author would like to thank the referee for his helpful advices and questions on the presentation of the paper. Thanks to his question, the second half of Section 7 appeared. This work was supported by JSPS KAKENHI Grant Number 15K04896.

\end{document}